\documentclass[
11pt,%
tightenlines,%
twoside,%
onecolumn,%
nofloats,%
nobibnotes,%
nofootinbib,%
superscriptaddress,%
noshowpacs,%
centertags]%
{revtex4}
\usepackage{ljm}

\usepackage{amsmath,latexsym,amssymb,amsfonts,amsthm,verbatim,enumerate,graphicx,color,bm,mathtools}

\usepackage{comment}

\usepackage{longtable}

\newtheorem{defi}{Definition}
\newtheorem{prop}[defi]{Proposition}
\newtheorem{lemm}[defi]{Lemma}
\newtheorem{theor}[defi]{Theorem}
\newtheorem{cor}[defi]{Corollary}
\newtheorem{rem}[defi]{Remark}

\def\squareforqed{\hbox{\rlap{$\sqcap$}$\sqcup$}}
\def\qed{\ifmmode\squareforqed\else{\unskip\nobreak\hfil
\penalty50\hskip1em\null\nobreak\hfil\squareforqed
\parfillskip=0pt\finalhyphendemerits=0\endgraf}\fi}
\def\endenv{\ifmmode\;\else{\unskip\nobreak\hfil
\penalty50\hskip1em\null\nobreak\hfil\;
\parfillskip=0pt\finalhyphendemerits=0\endgraf}\fi}
\mathchardef\ordinarycolon\mathcode`\:
\mathcode`\:=\string"8000
\def\vcentcolon{\mathrel{\mathop\ordinarycolon}}
\begingroup \catcode`\:=\active
  \lowercase{\endgroup
  \let :\vcentcolon
  }

\newcommand{\nc}{\newcommand}
\nc{\rnc}{\renewcommand}
\nc{\beq}{\begin{equation}}
\nc{\eeq}{{\end{equation}}}
\nc{\bea}{\begin{eqnarray}}
\nc{\eea}{\end{eqnarray}}
\nc{\beqa}{\begin{eqnarray}}
\nc{\eeqa}{\end{eqnarray}}
\nc{\lbar}[1]{\overline{#1}}
\nc{\bra}[1]{\langle#1|}
\nc{\ket}[1]{|#1\rangle}
\nc{\ketbra}[2]{|#1\rangle\!\langle#2|}
\nc{\braket}[2]{\langle#1|#2\rangle}
\nc{\proj}[1]{|#1\rangle\!\langle #1|}
\nc{\avg}[1]{\langle#1\rangle}
\rnc{\max}{\operatorname{max}}
\nc{\rank}{\operatorname{rank}}
\nc{\conv}{\operatorname{conv}}
\nc{\smfrac}[2]{\mbox{$\frac{#1}{#2}$}}
\nc{\Tr}{\operatorname{Tr}}
\nc{\ox}{\otimes}
\nc{\dg}{\dagger}
\nc{\dn}{\downarrow}

\nc{\cA}{{\cal A}} \nc{\cB}{{\cal B}}
\nc{\cC}{{\cal C}} \rnc{\cD}{{\cal D}} \nc{\cE}{{\cal E}}
\nc{\cF}{{\cal F}} \nc{\cG}{{\cal G}} \rnc{\cH}{{\cal H}}
\nc{\cI}{{\cal I}} \nc{\cJ}{{\cal J}} \nc{\cK}{{\cal K}}
\rnc{\cL}{{\cal L}} \nc{\cM}{{\cal M}} \nc{\cN}{{\cal N}}
\nc{\cO}{{\cal O}} \nc{\cP}{{\cal P}} \rnc{\cR}{{\cal R}}
\nc{\cS}{{\cal S}} \nc{\cT}{{\cal T}} \nc{\cU}{{\cal U}}
\nc{\cX}{{\cal X}} \nc{\cW}{{\cal W}} \nc{\cZ}{{\cal Z}}

\nc{\csupp}{{\operatorname{csupp}}}
\nc{\qsupp}{{\operatorname{qsupp}}} \nc{\var}{\operatorname{var}}
\nc{\rar}{\rightarrow} \nc{\lrar}{\longrightarrow}
\nc{\poly}{\operatorname{poly}}
\nc{\polylog}{\operatorname{polylog}}
\nc{\Lip}{\operatorname{Lip}}
\rnc{\1}{\openone}

\def\>{\rangle}
\def\<{\langle}

\def\a{\alpha}
\def\b{\beta}
\def\g{\gamma}

\def\t{\theta}

\def\s{\sigma}

\nc{\glneq}{{\raisebox{0.6ex}{$>$}  \hspace*{-1.8ex} \raisebox{-0.6ex}{$<$}}}
\nc{\gleq}{{\raisebox{0.6ex}{$\geq$}\hspace*{-1.8ex} \raisebox{-0.6ex}{$\leq$}}}

\nc{\RR}{{{\mathbb R}}}
\nc{\CC}{{{\mathbb C}}}
\nc{\FF}{{{\mathbb F}}}
\nc{\HH}{{{\mathbb H}}}
\nc{\NN}{{{\mathbb N}}}
\nc{\ZZ}{{{\mathbb Z}}}
\nc{\bZ}{{{\mathbb Z}}}

\nc{\PP}{{{\mathbb P}}}
\nc{\QQ}{{{\mathbb Q}}}
\nc{\bQ}{{{\mathbb Q}}}
\nc{\UU}{{{\mathbb U}}}
\nc{\bU}{{{\mathbb U}}}
\nc{\WW}{{{\mathbb W}}}
\nc{\EE}{{{\mathbb E}}}
\rnc{\SS}{{{\mathbb S}}}

\nc{\id}{{\operatorname{id}}}

\nc{\vholder}[1]{\rule{0pt}{#1}}

\nc{\ob}[1]{#1}

\def\beq{\begin{equation}}
\def\eeq{\end{equation}}

\nc{\eq}[1]{Eq.~(\ref{eq:#1})} \nc{\eqs}[2]{Eqs.~(\ref{eq:#1}) and
(\ref{eq:#2})}

\nc{\eqn}[1]{Eq.~(\ref{eqn:#1})}
\nc{\eqns}[2]{Eqs.~(\ref{eqn:#1}) and (\ref{eqn:#2})}

\nc{\region}{\cS\cW}

\makeatletter
\renewcommand*{\@fnsymbol}[1]{\ensuremath{\ifcase#1\or
\ast\or \P\or \natural\or \flat\or \sharp\or \bullet\or \dagger\or \ddagger\or
\|\or **\or \dagger\dagger\or \ddagger\ddagger \else\@ctrerr\fi}}
\makeatother

\begin{document}
\titlerunning{Geometry of $SO(3)_p$}
\authorrunning{S. Di Martino, S. Mancini, M. Pigliapochi, and I. Svampa {\&} A. Winter}

\title{Geometry of the $p$-adic special orthogonal group $SO(3)_p$
\footnote{The present work grew out of a BSc thesis
\cite{Ilaria:tesina} and two MSc theses
\cite{Michele:tesi,Ilaria:tesi} at the University of Camerino,
co-supervised at the Autonomous  University of
Barcelona.\vspace{2mm}}}
%\title{Structure and geometry of the $\mathbf{p}$-adic special orthogonal groups in dimensions 2 and 3}

\author{Sara Di Martino}
 \email[E-mail: ]{saradimarti@gmail.com}
 \affiliation{F\'{\i}sica Te\`{o}rica: Grup d'Informaci\'{o} Qu\`{a}ntica,
              Departament de F\'isica, %\protect\\
              Universitat Aut\`onoma de Barcelona, ES-08193 Bellaterra (Barcelona), Spain}

\author{Stefano Mancini}
 \email[E-mail: ]{stefano.mancini@unicam.it}
 \affiliation{School of Science and Technology, University of Camerino, %\protect\\
              Via Madonna delle Carceri 9, I-62032 Camerino, Italy}
 \affiliation{INFN --- Sezione Perugia, Via A. Pascoli, I-06123 Perugia, Italy}

\author{Michele~Pigliapochi}
 \email[E-mail: ]{michele.pigliapochi@gmail.com}
 \affiliation{School of Science and Technology, University of Camerino, %\protect\\
              Via Madonna delle Carceri 9, I-62032 Camerino, Italy}

\author{Ilaria Svampa}
 \email[E-mail: ]{ilaria.svampa@unicam.it}
 \affiliation{School of Science and Technology, University of Camerino, %\protect\\
              Via Madonna delle Carceri 9, I-62032 Camerino, Italy}
 \affiliation{F\'{\i}sica Te\`{o}rica: Grup d'Informaci\'{o} Qu\`{a}ntica,
              Departament de F\'isica, %\protect\\
              Universitat Aut\`onoma de Barcelona, ES-08193 Bellaterra (Barcelona), Spain}

\author{Andreas Winter}
\email[E-mail: ]{andreas.winter@uab.cat}
 \affiliation{F\'{\i}sica Te\`{o}rica: Grup d'Informaci\'{o} Qu\`{a}ntica,
              Departament de F\'isica, %\protect\\
              Universitat Aut\`onoma de Barcelona, ES-08193 Bellaterra (Barcelona), Spain}
 \affiliation{ICREA --- Instituci\'o Catalana de la Recerca i Estudis Avan\c{c}ats, %\protect\\
              Pg. Llu\'is Companys, 23, ES-08010 Barcelona, Spain}

%\date{13 January 2023}

%\firstcollaboration{(Submitted by G. G. Amosov)}

%\received{January 01, 2023; revised January 10, 2023; accepted January 20, 2023}

\begin{abstract}
  We derive explicitly the structural properties of the $p$-adic special
  orthogonal groups in dimension three, for all primes $p$,
  and, along the way, the two-dimensional case. In particular,
  starting from the unique definite quadratic form in three dimensions
  (up to linear equivalence and rescaling), we show that every element
  of $SO(3)_p$ is a rotation around an axis. An important part of the analysis
  is the classification of all definite forms in two dimensions, yielding
  a description of the rotation subgroups around any fixed axis, which
  all turn out to be abelian and parametrised naturally by the projective line.

  Furthermore, we find that for odd primes $p$, the entire group $SO(3)_p$
  admits a representation in terms of Cardano (aka nautical) angles of
  rotations around the reference axes, in close analogy to the real orthogonal case.
  However, this works only for certain orderings of the product of rotations around
  the coordinate axes, depending on the prime; furthermore, there is no
  general Euler angle decomposition. For $p=2$, no Euler or Cardano decomposition
  exists.
\end{abstract}



\maketitle

\section{Introduction}
\label{sec:intro} Large parts of classical geometry can be
interpreted, following Klein's programme \cite{Klein}, as statements
about the orthogonal and special orthogonal groups in dimensions
$2$, $3$, and generally $n$. It is thus no surprise that the special
orthogonal groups $SO(n)_{\RR}$ are among the most studied and most
well-known groups in mathematics. The cases $n=2$ and $n=3$ stand
out particularly, both for their structural properties and the
possibility to visualise the action of the group on Euclidean space.
In particular, $SO(2)_{\RR}$ is the group of planar rotations,
i.e.,~isomorphic to the commutative group of adding angles mod
$2\pi$; and $SO(3)_{\RR}$ consists entirely of rotations around
different axis, admitting essentially unique decompositions into
Euler and Cardano angles.

Given the importance of $p$-adic numbers in number theory, it is
natural that orthogonal groups should have been studied also over
the fields $\QQ_p$, which a priori are a multitude of symmetry
groups, one for each nontrivial quadratic form. Just as in the real
case, a fundamental property of the quadratic form is whether it is
definite (i.e.,~only has the trivial zero) or indefinite
(i.e.,~represents zero in nontrivial ways), which distinguishes
compact symmetry groups in the latter and non-compact groups in the
former case. Unlike the real case, definite quadratic forms over
$\QQ_p$ exist only in dimensions two, three and four. As we shall
see here, in dimensions two and three the symmetry groups, denoted
by $SO(2)_p$ and $SO(3)_p$, are largely governed by structures
familiar from Euclidean geometry, reinterpreted $p$-adically.

The structure of the paper is as follows. In Section
\ref{sec:quadratic}, we review and re-derive the unique quadratic
forms on dimensions three and four that do no represent $0$
nontrivially (``definite'' forms), up to linear and rescaling
equivalence, for all primes $p$, introducing some useful notation
for the rest of the paper. This material, albeit well-known, is
included for the sake of a self-contained exposition. This
classification allows us to define the $p$-adic special orthogonal
groups $SO(3)_p$ and $SO(4)_p$ in a unique way, though we will not
consider the latter afterwards in the present work.

Then, in Section \ref{sec:basic}, we begin our new contributions, by
first making several basic observations about $SO(3)_p$, most importantly
that it is compact and profinite, and that every of its elements is a
rotation around some axis in $p$-adic three-space $\QQ_p^3$.
This then motivates the investigation of the rotations around a fixed given axis,
which are special orthogonal transformations of the plane orthogonal to that axis,
in Sections \ref{sec:planar} and \ref{gendir}.
They are naturally the orthogonal symmetry groups of the
definite form restricted to a plane, and there
are three ($p$ odd), resp.~seven ($p=2$) equivalence classes of those.
We denote them $SO(2)_p^\kappa$, and they all turn out to be abelian groups.
Furthermore, we derive a parametrisation of each of these rotation groups
by the $p$-adic projective line $P^1(\QQ_p)$, allowing us to identify the
groups $SO(2)_p^\kappa$ with certain abelian subgroups of $PGL(2,\QQ_p)$,
illuminating in particular the composition law.
Finally, in Section \ref{sec:CardEul2}, we show that every element of $SO(3)_p$
has a decomposition into Cardano principal angles, for odd primes $p$:
there exist orderings of the coordinate axes such that every special orthogonal
group element is a product of rotations around the axes in that order
(in exactly two different ways). For $p=2$, however, no fixed ordering of
the product can recover the entire group $SO(3)_2$, and we exhibit an example
of a special orthogonal matrix that cannot be written in any of the twelve
possible Euler or Cardano ways.

\section{Quadratic forms over $p$-adic numbers}
\label{sec:quadratic}
In this section we review, and in some cases, re-derive the elementary properties
we shall need about the quadratic forms over $\QQ_p$, according to dimension
($n=3$ and $n=4$) as well as type of prime ($p\equiv 1\mod 4$, $p\equiv 3 \mod 4$
and $p=2$). Comprehensive treatments of this material can be found in
the books of Cassels \cite{Cassels} and Serre \cite{Serre}.

A quadratic form for our purposes is a homogeneous function on the
$n$-dimensional $\QQ_p$-vector space $V$ that can be written as
\begin{equation}\notag
  Q(\boldsymbol{x})=\sum_{ij} a_{ij}x_ix_j = x^\top A x,
\end{equation}
where $\boldsymbol{x}=\sum_i x_i \boldsymbol{e}_i\in V$ is a vector,
$(\boldsymbol{e}_1,\ldots,\boldsymbol{e}_n)$ is a basis of $V$, and
$A$ is an $n\times n$-matrix. Throughout we will assume that $A$ is
nondegenerate, i.e., has rank $n$. Equivalently, we could speak
about symmetric bilinear forms $b(\boldsymbol{x},\boldsymbol{y})$ on
$V$ such that $Q(\boldsymbol{x})=b(\boldsymbol{x},\boldsymbol{x})$,
since we recover the bilinear form via
$b(\boldsymbol{x},\boldsymbol{y})=
\frac12\bigl(Q(\boldsymbol{x}+\boldsymbol{y})-Q(\boldsymbol{x})-Q(\boldsymbol{y})\bigr)$.
In either case, the orthogonal group is defined as the set of linear
maps on $V$ that are symmetries of the quadratic/bilinear form
\[\begin{split}
  O(Q) &= \{ L\in \text{End}(V) : Q(L\boldsymbol{x}) = Q(\boldsymbol{x}) \ \forall \boldsymbol{x}\in V \} \\
       &= \{ L\in \text{End}(V) : b(L\boldsymbol{x},L\boldsymbol{y}) = b(\boldsymbol{x},\boldsymbol{y})
                                                          \ \forall \boldsymbol{x},\boldsymbol{y}\in V \} \\
       &\simeq \{ L\in M_{n\times n}(\QQ_p) : L^\top A L = A \},
\end{split}\]
the latter under the identification of $V$ with $\QQ_p^n$ via the basis
$(\boldsymbol{e}_i)$,
$\boldsymbol{x}=\sum_i x_i \boldsymbol{e}_i \leftrightarrow (x_1,\ldots,x_n)$, which
turns the linear maps on $V$ into $n\times n$-matrices. The subset of
$O(Q)$ consisting of matrices $L$ with unit determinant, $\det L = 1$,
is the special orthogonal group, denoted $SO(Q)$.

We are interested in the abstract group structure of $O(Q)$ and $SO(Q)$,
which do not change when going to an equivalent form.
First, $Q'$ is \emph{similar} to $Q$, $Q\sim Q'$,
if there exists an invertible linear map $S$ such that
$Q'(\boldsymbol{x})=Q(S\boldsymbol{x})$ for all $\boldsymbol{x}\in V$,
meaning for the matrix representation $A'$
of $Q'$ that $A' = S^\top A S$. In that case, $O(Q') \simeq O(Q)$ and
$SO(Q') \simeq SO(Q)$, the isomorphism being $O(Q) \ni L \mapsto S^{-1}LS \in O(Q')$.
Furthermore, $Q'$ is a \emph{scaling} of $Q$ if $Q'=tQ$ with $t\in \QQ_p^*$;
in this case, clearly $O(Q') = O(Q)$ and $SO(Q') = SO(Q)$.
Hence, our first task is to classify the quadratic forms up to
similarity and scaling, which we sum up into \emph{equivalence}.
Indeed, up to similarity, we can write every quadratic form with a
diagonal matrix $A$, i.e.,
\begin{equation}\notag
  Q(\boldsymbol{x})=\sum_j a_j x_j^2,
\end{equation}
with $\boldsymbol{x}=(x_1,x_2,\dots,x_n)\in\QQ_p^n$, and $a_j \in \QQ_p^*$.
Multiplying $x_j$ by $\lambda_j^{-1}\neq 0$ will change $a_j$ to $\lambda_j^2 a_j$,
hence to classify the quadratic forms up to
coordinate changes $GL_n(\QQ_p)$, we only need to consider $a_j\in\QQ_p^*/(\QQ_p^*)^2$.
We thus have to understand the structure of the group $K=\QQ_p^*/(\QQ_p^*)^2$.
The following descriptions of $K$ are well-known~\cite{Cassels}:

%\medskip\noindent
\textbf{If $\mathbf{p\neq 2}$}, then $K=\langle u, p\rangle = \{1,u,p,up\}$,
with $u$ a unit in $\ZZ_p$ that is not a square. Clearly $K \simeq (\ZZ/2\ZZ)^2$
is the Klein group.
For $p\equiv 3\mod 4$, we may choose $u=-1$.

%\medskip\noindent
\textbf{If $\mathbf{p=2}$}, then $K=\langle -1,2,5\rangle =\{1,-1,2,-2,5,-5,10,-10\}$.
In this case $K \simeq (\ZZ/2\ZZ)^3$.

%\medskip
Given the structure of $K$, it remains to find the invariants of quadratic forms. It can be proven~\cite{Cassels} that apart from the rank of the form, there are only two more
invariants: the discriminant $d(Q)=\Pi_j a_j=\det A$, as well as $\varepsilon(Q)=\Pi_{j<k}(a_j,a_k)$,
where $(a,b)$ is the Hilbert symbol defined as
\begin{equation}\notag
  (a,b):= \begin{cases}
            \phantom{-}1 & \text{ iff } z^2-ax^2-by^2=0 \text{ admits nontrivial solutions,}\\
                      -1 & \text{ otherwise.}
          \end{cases}
\end{equation}

Now we have all the information we need in order to classify
quadratic forms on $\QQ_p^n$ for all $n$.

\begin{theor}
\label{teor:equivpQ}
Two quadratic forms over $\QQ_p$ are similar if and only if they have same
rank $n$, same determinant $d$ and same Hasse invariant $\varepsilon$.
{\hfill\qed}
\end{theor}

\begin{theor}
\label{teor:Qrepns0}
The quadratic form $Q$ on $\QQ_p^n$ represents $0$ nontrivially if and only if
\begin{itemize}
    \item $n=2$ and $d\simeq-1$ in $K$;
    \item $n=3$ and $\varepsilon=(-1,-d)$;
    \item $n=4$ and either $d\not\simeq1$ or $d\simeq1$ and $\varepsilon=(-1,-1)$;
    \item $n\geq5$. {\hfill\qed}
\end{itemize}
\end{theor}

We record explicitly the quadratic forms for $n=3$ up to equivalence,
separate by odd and even primes $p$.

%\medskip\noindent
\textbf{Prime $\mathbf{p}$ odd:}
There are exactly two inequivalent forms on $\QQ_p^3$,
\begin{align*}
  Q_+(\boldsymbol{x}) &= x_1^2 - v x_2^2 + p x_3^2, \\
  Q_0(\boldsymbol{x}) &= x_1^2 + x_2^2 + x_3^2,
\end{align*}
where
\begin{equation}
  \label{eq:woo}
  v=\begin{cases}
      -u & \text{ if } p\equiv1 \mod 4,\\
      -1 & \text{ if } p\equiv3 \mod 4,
    \end{cases}
\end{equation}
is a particular choice of a non-square in $\QQ_p$.

%\medskip\noindent
\textbf{Prime $\mathbf{p=2}$: }
There are exactly two inequivalent forms on $\QQ_2^3$,
\begin{align*}
  Q_+(\boldsymbol{x}) &= x_1^2 + x_2^2 + x_3^2, \\
  Q_0(\boldsymbol{x}) &= x_1^2 + x_2^2 - x_3^2.
\end{align*}

%\medskip\noindent
\textbf{Real Euclidean case}:
Also here, there are exactly two inequivalent forms on $\RR^3$,
\begin{align*}
  Q_+^\RR(\boldsymbol{x}) &= x_1^2 + x_2^2 + x_3^2, \\
  Q_0^\RR(\boldsymbol{x}) &= x_1^2 + x_2^2 - x_3^2.
\end{align*}

%\medskip
In all cases, the real and the $p$-adic ones, one form ($Q_+$) is
definite, in the sense that it does not represent zero nontrivially
(i.e., $Q_+(\boldsymbol{x})=0$ iff $\boldsymbol{x}=0$), while the
other ($Q_0$) is indefinite, in the sense that it has isotropic
vectors (i.e., $Q_0(\boldsymbol{x})=0$ for some $\boldsymbol{x}\neq
0$). The symmetry group $SO(Q_0)$ is always non-compact, as is
well-known in the real case and easy to see in general in the
$p$-adic case, since it has a hyperbolic component
\cite{GrossReeder}. On the other hand, in the real Euclidean case,
$SO(Q_+)$ is just the real $SO(3)_{\RR}$, which is a compact Lie
group, and as we would like to preserve the compactness in the
$p$-adic case, we define $SO(3)_p = SO(Q_+)$ for all primes $p$.
According to the above classification, this is a unique and
well-defined group for every prime $p$, which will indeed turn out
to be compact.

%\medskip
For $n=4$, we similarly focus only on the definite forms. It turns
out that again, for $\RR$ and all $\QQ_p$, there is exactly one up
to equivalence
\begin{equation}
  \notag
  Q_+^{(4)}(\boldsymbol{x})
         = \begin{cases}
             x_1^2 - v x_2^2 + p x_3^2 - vp x_4^2 & \text{ if } p \text{ odd}, \\
             x_1^2 + x_2^2 + x_3^2 + x_4^2        & \text{ if } p=2,           \\
             x_1^2 + x_2^2 + x_3^2 + x_4^2        & \text{ in the real case}.
           \end{cases}
\end{equation}
All other inequivalent forms represent zero nontrivially (and hence their
associated orthogonal groups are non-compact \cite{GrossReeder}),
a fact that we will exploit later.
This means, that also for $n=4$, both $\QQ_p^4$ and $\RR^4$ have essentially
unique groups $SO(4)_p$ and $SO(4)_{\RR}$. On the other hand, for $n\geq 5$,
it is known that for all primes $p$, there are no definite quadratic forms
on $\QQ_p^n$ \cite{Cassels,Serre}, unlike the real Euclidean case $\RR^n$,
where of course the sum of the squares of the standard coordinates is the
unique definite form up to equivalence, and where we thus have an unambiguous
compact Lie group $SO(n)_{\RR}$. Because of the lack of definite quadratic
forms for $n\geq5$, we refrain from speaking of $SO(n)_p$ without qualification.

Even in the real Euclidean case, $n=3$ is distinguished by several
geometric peculiarities. We highlight two of them, see
\cite[Ch.~4]{Eulerangle}.

\begin{theor}
\label{teor:dicougualeinP}
Elements of $SO(3)_{\RR}$ are rotations about axes in $\RR^3$, meaning that they
always have eigenvalue $1$, for which the corresponding eigenspace is the rotation axis.
{\hfill\qed}
\end{theor}

\begin{theor}
\label{teor:EulercompRot} With $R_x(\theta)$, $R_y(\eta)$,
$R_z(\phi)$ denoting the rotations around the reference axes of
$\RR^3$ by a given angle $\theta$, $\eta$ and $\phi$, every $R\in
SO(3)$ can be written as a product of three such rotations in any of
the following forms
 \beq\notag
\begin{array}{cccccc}
     R_xR_yR_z, &R_yR_zR_x, &R_zR_xR_y, &R_xR_zR_y, &R_zR_yR_x, &R_yR_xR_z, \\
     R_xR_yR_x, &R_xR_zR_x, &R_yR_xR_y, &R_yR_zR_y, &R_zR_xR_z, &R_zR_yR_z.
\end{array}
\eeq These angles are called \emph{Cardano} or \emph{Tait--Bryan
angles}, and also nautical angles, when applied to a form from the
first row, and (proper) \emph{Euler angles} when applied to a form
from the second row.

In both cases, the choice of the angles is unique (up to isolated points) modulo $2\pi$
radians for $\theta$ and $\phi$, and with the range of $\eta$ covering $\pi$ radians.
{\hfill\qed}
\end{theor}

This theorem implies that three successive rotations relative to coordinate axes
generate every rotation in $\RR^3$. Hence, $SO(3)_{\RR}$ is generated by the rotations
around the three reference axes of $\RR^3$; in the case of the Euler decomposition,
actually only two reference axes.
Each of these rotation groups is an infinite cyclic Lie group, so $SO(3)_{\RR}$ is
generated by two cyclic subgroups linked by a non-commutative relation.

%\medskip
We have left out $n=2$ until now, which may seem strange coming from the
real Euclidean case, where it is the same as for all other $n$: there is a
unique definite quadratic form on $\RR^2$. Also on $\QQ_p^2$ there are
definite quadratic forms, but now they are not unique up to equivalence.
We will come back to them in detail below.

\section{Basic observations about $SO(3)_p$}
\label{sec:basic}
We start by deriving a few basic facts about the $p$-adic special orthogonal
groups in dimension $3$ from the definition. It is evidently a group under the
usual matrix multiplication. In fact, it is a topological group, with the operations
of multiplication and inverse being continuous with respect to the $p$-adic metric. Explicitly, the topology on $SO(3)_p$ is the one generated by the open balls with respect to the $p$-adic norm $||L||_p=||(\ell_{ij})_{ij}||_p\coloneqq \max_{i,j=1,2,3}|\ell_{ij}|_p$, where $|\,\cdot\,|_p$ denotes the $p$-adic absolute value on $\mathbb{Q}_p$.

\begin{theor}
\label{thm:compactness}
For every prime $p$, the group $SO(3)_p$ is compact. As a matter of fact,
it is a closed subset (with respect to the $p$-adic metric) of
$M_{3\times 3}(\ZZ_p)$, the set of matrices with $p$-adic integer entries,
and so $SO(3)_p \subset SL(3,\ZZ_p)$.
\end{theor}
\begin{proof}
Let $L=(\ell_{ij})_{ij} \in SO(3)_p$, and write $\ell_{ij}=p^{\nu_{ij}}u_{ij}$,
where $\nu_{ij}=\nu_p(\ell_{ij})\in\ZZ\cup\{+\infty\}$
is the $p$-adic valuation of $\ell_{ij}$, and $u_{ij}\in \UU_p$ is a
unit in $\ZZ_p$. We need to show that $\nu_{ij}\geq 0$ for all $i,j\in\{1,2,3\}$.

%\medskip
When $p$ is odd, the defining condition $A=L^\top AL$ of $SO(3)_p$
implies the following three relations
\begin{equation}\begin{split}
  \label{eq:sisconddiagdef}
  \ell_{11}^2-v\ell^2_{21}+p\ell_{31}^2 &=  1,\\
  \ell_{12}^2-v\ell^2_{22}+p\ell_{32}^2 &= -v,\\
  \ell_{13}^2-v\ell^2_{23}+p\ell_{33}^2 &=  p.
\end{split}\end{equation}
The first one is equivalent to
\beq\notag
  p^{2\nu_{11}}u_{11}^2-vp^{2\nu_{21}}u_{21}^2+p^{1+2\nu_{31}}u_{31}^2 = 1.
\eeq
Let us assume by contradiction that $\min\{\nu_{11},\nu_{21},\nu_{31}\}<0$.
If $\min\{2\nu_{11},\,2\nu_{21},1+2\nu_{31}\}=2\nu_{11}<0$ (similarly for $\nu_{21}$),
we multiply both sides of the equation by $p^{2\lvert \nu_{11}\rvert}$ to obtain
\beq\notag
  u_{11}^2-vp^{2(\nu_{21}-\nu_{11})}u_{21}^2+p^{1+2(\nu_{31}-\nu_{11})}u_{31}^2 = p^{2\lvert \nu_{11}\rvert},
\eeq
where every term is a $p$-adic integer. Hence, for every $k=1,\dots, 2\lvert \nu_{11}\rvert$,
\beq\notag
  u_{11}^2-vp^{2(\nu_{21}-\nu_{11})}u_{21}^2+p^{1+2(\nu_{31}-\nu_{11})}u_{31}^2 \equiv 0 \mod p^k,
\eeq
and in particular $u_{11}^2-vp^{2(\nu_{21}-\nu_{11})}u_{21}^2\equiv 0 \mod p$.
The quadratic form $x^2-vy^2$ does not represent $0$ in $\QQ_p$:
if $x,y\in\ZZ_p$ as in our case, the only solution to $x^2-vy^2\equiv 0 \mod p$ is
$(0,0)\in(\ZZ/p\ZZ)^2$. This gives a contradiction, since $u_{11}\in\UU_p$, in particular
$u_{11}^2\not\equiv 0\mod p$.

If instead $\min\{2\nu_{11},\,2\nu_{21},1+2\nu_{31}\}=1+2\nu_{31}<0$, we rewrite the
equation in terms of $p$-adic integers as
\beq\notag
  p^{2(\nu_{11}-\nu_{31})-1}u_{11}^2-vp^{2(\nu_{21}-\nu_{31})-1}u_{21}^2+u_{31}^2
      = p^{2\lvert\nu_{31}\rvert-1}.
\eeq
Reducing it modulo $p$ we get $u_{31}^2 \equiv 0\mod p$, which is in contradiction with
the hypothesis that $u_{31}\in\UU_p$.

The same can be done for the second and third equations of Eq.~\eqref{eq:sisconddiagdef},
since $-v,\,p\in\ZZ_p$.

%\medskip
When $p=2$, the defining condition $L^\top L = I$ implies
\beq\notag
  2^{2\nu_{1i}}u_{1i}^2+2^{2\nu_{2i}}u_{2i}^2+2^{2\nu_{3i}}u_{3i}^2=1, \text{ for } i=1,2,3.
\eeq
Again assuming by contradiction $\min\{\nu_{1i},\nu_{2i},\nu_{3i}\}=\nu_{1i}<0$
(the latter without loss of generality, by symmetry),
we rewrite the last equation in terms of $2$-adic integers as
\beq\notag
  u_{1i}^2+2^{2(\nu_{2i}-\nu_{1i})}u_{2i}^2+2^{2(\nu_{3i}-\nu_{1i})}u_{3i}^2 = 2^{2\lvert\nu_{1i}\rvert}.
\eeq
As a consequence it must hold
\beq\notag
  u_{1i}^2+2^{2(\nu_{2i}-\nu_{1i})}u_{2i}^2+2^{2(\nu_{3i}-\nu_{1i})}u_{3i}^2
       \equiv 0\mod 2^k,\ k=1,\dots, 2\lvert\nu_{1i}\rvert.
\eeq
The quadratic form $x^2+y^2+z^2$ does not represent $0$ in $\QQ_2$: as a matter of
fact, as $x,y,z$ are in $\ZZ_2$, it does not so in $\ZZ/4\ZZ$, unless
all three variables are $\equiv 0 \mod 4$; but
this is again in contradiction to $u_{1i}\in\UU_p$.

We have found $SO(3)_p\subset M_{3\times3}(\ZZ_p)$,
and the defining condition $\det L = 1$ implies $SO(3)_p\subset SL(3,\ZZ_p)$.
\end{proof}

\begin{theor}
\label{thm:rotation} All elements of $SO(3)_p$ are rotations,
i.e.,~they always have an eigenvalue $1$, and the corresponding
eigenspace is the rotation axis.
\end{theor}
\begin{proof}
Let $L\in SO(3)_p$, if  $\lambda$ is an eigenvalue of $L$, then also $\lambda^{-1}$ is.
In fact if $x$ is an eigenvector for $\lambda$, we have
$Ax=L^\top ALx=\lambda L^\top Ax\Rightarrow \lambda^{-1}Ax=L^\top Ax$. In other words,
$\lambda^{-1}$ is an eigenvalue of $L^\top $ with eigenvector $Ax$.
On the other hand, $L$ and $L^\top $ share the characteristic polynomial and
hence the eigenvalues.

Suppose now that none of the eigenvalues $\lambda_i$ of $L$, for $i=1,2,3$, are equal
to $1$. Then, $\lambda_i=\lambda_i^{-1}$ for all $i$ (equivalently
$\lambda_i^2=1$, thus $\lambda_i=\pm 1$), otherwise we can suppose,
for example, $\lambda_1\neq \lambda_1^{-1}=\lambda_2$ without loss
of generality, and hence $1=\det
L=\lambda_1\lambda_2\lambda_3=\lambda_3$, which is a contradiction.

Summing up, the only case which allows the condition proved before, excluding the
presence of $1$ among the eigenvalues, is $\lambda_1=\lambda_2=\lambda_3=-1$,
which is in contradiction with the condition $\det L=1$.
Thus, $L\in SO(3)_p$ has always $1$ as eigenvalue.
\end{proof}

%\medskip
Theorem \ref{thm:compactness} means that we have well-defined projection maps
\begin{align*}
  \pi_k: SO(3)_p            &\longrightarrow SO(3)_p \mod p^k \subset SL(3,\ZZ/p^k\ZZ), \nonumber\\
         L=(\ell_{ij})_{ij} &\longmapsto     (\ell_{ij} \mod p^k)_{ij},
\end{align*}
since all matrix entries of eligible $L$ are $p$-adic integers.
The images $SO(3)_p \mod p^k = \pi_k\bigl(SO(3)_p\bigr)$ are all
finite groups, forming a projective system under the (commuting)
projection maps of taking a number in $\ZZ/p^{k'}\ZZ$ modulo $p^k$
($k<k'$), which by slight abuse of notation we denote by $\pi_k$, too.
Just as the $p$-adic integers $\ZZ_p$ are the inverse (aka projective)
limit of the rings $\ZZ/p^k\ZZ$ connected by the modulo $p^k$ projections,
we conclude that $SO(3)_p$ is the inverse limit of the finite groups
$SO(3)_p \mod p^k$ connected by the projections $\pi_k$. This makes
$SO(3)_p$ a so-called \emph{profinite} group.

%\medskip
Theorem \ref{thm:rotation} means that we can describe arbitrary element $L\in SO(3)_p$
by first picking an axis of rotation, $\QQ_p \mathbf{n}$, with a nonzero
vector $\mathbf{n}\in\QQ_p^3$, construct the two-dimensional subspace
$V = \mathbf{n}^\perp = \{x : b(\mathbf{n},x) = 0\} \subset \QQ_p^3$
and consider the quadratic form $Q_{+\vert V}$, which is necessarily definite.
Thus, $L$ can be written as $L = L_V + b(\mathbf{n},\cdot)\frac{1}{Q_+(\mathbf{n})}\mathbf{n}$,
with a two-dimensional special orthogonal transformation $L_V \in SO(Q_{+\vert V})$.
In the next sections, we shall consider what forms can appear as restrictions
$Q_{+\vert V}$, and which are the different appearances of $SO(2)_p$.

\begin{lemm}
\label{prop:trasfortbasis}
For any two orthogonal bases
$\mathcal{B}=(\boldsymbol{v}_1,\boldsymbol{v}_2,\boldsymbol{v}_3)$ and
$\mathcal{C}=(\boldsymbol{w}_1,\boldsymbol{w}_2,\boldsymbol{w}_3)$ of $\QQ_p^3$,
there exists an orthogonal transformation $M:\bQ_p^3\to\bQ_p^3$ such that
$M\boldsymbol{v}_i=\boldsymbol{w}_i$ for all $i=1,2,3$ if and only if
$Q_+(\boldsymbol{v}_i)=Q_+(\boldsymbol{w}_i)$.

In that case, either $M$ is special or the orthogonal transformation $M'$
sending $\mathcal{B}$ to $\mathcal{C}'=(\boldsymbol{w}_1,-\boldsymbol{w}_2,\boldsymbol{w}_3)$
is special.
%This holds if we switch the sign to any of the three vectors of $\mathcal{C}$,
%or to all of them because $\det(-O)=(-1)^3\det O$.
\end{lemm}
\begin{proof}
There exists a unique linear transformation $M:\bQ_p^3\to\bQ_p^3$ sending $
\mathcal{B}$ to $\mathcal{C}$. $M$ is an orthogonal transformation if and only if
it preserves the bilinear form on $\bQ_p^3$ associated to $Q_+$.
It is enough to show $b(\boldsymbol{v}_i,\boldsymbol{v}_j)=b(\boldsymbol{w}_i,\boldsymbol{w}_j)$
for all $i,j\in\{1,2,3\}$. For $i\neq j$ this holds as we assume orthogonal bases.
For $i=j$ it amounts to $Q_+(\boldsymbol{v}_i)=Q_+(\boldsymbol{w}_i)$ for every $i\in\{1,2,3\}$.

$\mathcal{C}'$ is an orthogonal basis like $\mathcal{C}$, and the orthogonal
transformation $M$ exists if and only if the orthogonal $M'$ exists.
An orthogonal transformation has determinant $\pm1$, and as $\det M' = -\det M$,
exactly one of the two will have determinant $1$.
\end{proof}

In contrast to $\bQ_p^3$, every vector in $\RR^3$ can be normalized to $1$, as
$Q_+(\boldsymbol{v})\simeq 1$ modulo squares for all $\boldsymbol{v}\in\RR^3\setminus \boldsymbol{0}$ (indeed $\RR^\ast/(\RR^\ast)^2\simeq\{\pm1\}$ and $Q_+^\RR(\boldsymbol{v})>0$ for every $\boldsymbol{v}\in\RR^3\setminus \boldsymbol{0}$).
As a consequence, every orthogonal basis of $\RR^3$ can be made orthonormal, and
there always exists a proper or improper rotation mapping any orthonormal basis
to any other orthonormal basis. On the other hand, in $\QQ_p^3$ we can only consider
orthogonal bases due to the existence of different kinds of vectors, depending on the value in $K=\QQ_p^\ast/(\QQ_p^\ast)^2$ that the quadratic form $Q_+$ takes on the vectors of $\QQ_p^3\setminus \boldsymbol{0}$. In this case the condition of preservation of the quadratic form becomes relevant.

\begin{prop}
\label{prop:exrotiff}
Given $\boldsymbol{n}\in\bQ_p^3\setminus \boldsymbol{0}$, for any non-zero vectors
$\boldsymbol{v},\,\boldsymbol{w}\in\boldsymbol{n}^\perp\subset\bQ_p^3$
there exists a rotation $\cR_{\boldsymbol{n}}\in SO(3)_p$ such that
$\cR_{\boldsymbol{n}}\boldsymbol{v}=\boldsymbol{w}$ if and only if
$Q_+(\boldsymbol{v})=Q_+(\boldsymbol{w})$.
\end{prop}
\begin{proof}
The direct implication is trivial by definition of $SO(3)_p$.

Conversely, for the orthogonal pair $\boldsymbol{v},\boldsymbol{n}$ there exists a unique vector
$\boldsymbol{v}'\in\bQ_p^3\setminus \boldsymbol{0}$, up to nonzero scalar multiples,
which completes $\boldsymbol{v},\boldsymbol{n}$ to an orthogonal basis
$\mathcal{B}=(\boldsymbol{v},\boldsymbol{v}',\boldsymbol{n})$.
Similarly, we can complete the pair $\boldsymbol{w},\boldsymbol{n}$ to an orthogonal
basis  $\mathcal{C}=(\boldsymbol{w},\boldsymbol{w}',\boldsymbol{n})$.

Now we prove that $Q_+(\boldsymbol{v}')$ and $Q_+(\boldsymbol{w}')$ are
in the same class in $K$. Indeed, $Q_+$ has diagonal matrix
representation on any orthogonal basis. We write the matrix
representation of $Q_+$ on the basis $\mathcal{B}$ as
$M=\text{diag}(q_1,q_2,q_3)$ for every
$\boldsymbol{v}'\in\{\boldsymbol{v},\boldsymbol{n}\}^\perp$ and on
$\mathcal{C}$ as $M'=\text{diag}(q_1',q_2',q_3')$ for every
$\boldsymbol{w}'\in\{\boldsymbol{w},\boldsymbol{n}\}^\perp$. The
hypothesis $Q_+(\boldsymbol{v})=Q_+(\boldsymbol{w})$ translates into
$q_1'=q_1$, and $q_3'=q_3$, the latter since $\mathcal{B}$ and
$\mathcal{C}$ share the third vector $\boldsymbol{n}$. The
determinant is an invariant of $Q_+$ in $K$, $\det M\simeq\det M'$: we
have $q_1q_2q_3=q_1q_2'q_3\,\lambda^2$ for some
$\lambda\in\bQ_p^\ast$, which gives $q_2=q_2'\,\lambda^2$, i.e.,
$Q_+(\boldsymbol{v}')=Q_+(\boldsymbol{w}')\lambda^2=Q_+(\lambda\boldsymbol{w}')$.

Therefore, there exists $\tilde{\boldsymbol{w}}'=\lambda\boldsymbol{w}'$ such that
$\tilde{\mathcal{C}}=(\boldsymbol{w},\tilde{\boldsymbol{w}}',\boldsymbol{n})$ is
an orthogonal basis for $\bQ_p^3$ and $Q_+(\boldsymbol{v}')=Q_+(\tilde{\boldsymbol{w}}')$.
To this basis we apply Lemma \ref{prop:trasfortbasis}: there exists a special
orthogonal transformation $\cR_{\boldsymbol{n}}:\bQ_p^3\to\bQ_p^3$
such that $\cR_{\boldsymbol{n}}\boldsymbol{v}=\boldsymbol{w}$,
$\cR_{\boldsymbol{n}}\boldsymbol{v}'=\tilde{\boldsymbol{w}}'$ and
$\cR_{\boldsymbol{n}}\boldsymbol{n}=\boldsymbol{n}$. % if and only if $Q(\boldsymbol{v})=Q(\boldsymbol{w})$.
In fact, either the orthogonal transformation $\cR_{\boldsymbol{n}}$ is already special,
or we choose $\tilde{\boldsymbol{w}}'=-\lambda\boldsymbol{w}'$ rather than
$\tilde{\boldsymbol{w}}'=\lambda\boldsymbol{w}'$ in the basis $\tilde{\mathcal{C}}$.
\end{proof}

We remark that there exists a transformation in $SO(3)_p$ rotating the
line $\QQ_p\boldsymbol{v}$ to the line $\QQ_p\boldsymbol{w}$ if and only
if $Q_+(\boldsymbol{v})\simeq Q_+(\boldsymbol{w})$. This is trivially satisfied in $\RR^3$,
where any line can be rotated to any other line. However, the necessary
condition of preservation of $Q_+$  implies that there does not always exist
a matrix in $SO(3)_p$ rotating a direction in $\bQ_p^3$ to another one.
For example, there is no element of $SO(3)_p$ for $p$ odd, rotating the $x$-axis of
$\bQ_p^3$ to the $z$-axis: $Q_+(\mu\boldsymbol{e}_1)=\mu^2\simeq 1$ and
$Q_+(\lambda \boldsymbol{e}_3)=\lambda^2 p\simeq p$ belong to different classes in $K$
for every $\mu,\lambda\in\bQ_p^\ast$.

The reverse implication of the previous proposition implies the following.

\begin{prop}
\label{prop:transitive}
Given a direction $\boldsymbol{n}\in\bQ_p^3$, the action of the subgroup of
$SO(3)_p$ of matrix rotations around $\boldsymbol{n}$ is transitive on the
equivalence classes $Q_+^{-1}(q)=\{\boldsymbol{v}\in\bQ_p^3\,:\,Q_+(\boldsymbol{v})=q\}$.

More generally, for any two vectors
$\boldsymbol{v},\boldsymbol{w}\in \QQ_p^3\setminus \boldsymbol{0}$
with $Q_+(\boldsymbol{v}) = Q_+(\boldsymbol{w})$, there exists an $L\in SO(3)_p$
with $L\boldsymbol{v} = \boldsymbol{w}$. This is achieved by choosing
$\boldsymbol{n} \in \{\boldsymbol{v},\boldsymbol{w}\}^\perp$.
{\hfill\qed}
\end{prop}

As noted above, there are different classes of vectors in
$\mathbb{Q}_p^3$, i.e., $\{\boldsymbol{v}\in
\mathbb{Q}_p^3\backslash\boldsymbol{0}\,:\,Q_+(\boldsymbol{v})\simeq
k\}$ for $k\in\mathbb{Q}_p^\ast/(\mathbb{Q}_p^\ast)^2$. This is one
of the fundamental obstacles in the study of $SO(3)_p$, which marks
a first departure from the real case.

\section{Planar case: the three (seven) incarnations of $SO(2)_p$}
\label{sec:planar}
By virtue of Theorem \ref{thm:rotation}, the elements of $SO(3)_p$ are transformations
occurring in the plane $V$ orthogonal to the rotation axis $\boldsymbol{n}$ in $\bQ_p^3$.
Thus we move to classify the definite quadratic forms on $\bQ_p^2$ and analyze the
corresponding symmetry groups.
We will derive a parameterization for the rotations in these groups and for $SO(3)_p$,
and show that the restrictions of the rotations of $SO(3)_p$ to the plane orthogonal to
their rotation axis realize all the classes of rotation groups of the $p$-adic plane.
Indeed, we will show that the restrictions $Q_{+\vert V}$ of the three-dimensional
definite form $Q_+$, where $V\subset \QQ_p^3$ varies over all two-dimensional
subspaces, exhaust all two-dimensional definite quadratic forms, up to equivalence
(similarity and scaling).

Beginning with odd primes $p$, the invariants of the two-dimensional
quadratic forms $Q(x,y)=ax^2+by^2$, $x,y\in\QQ_p$, characterized by
pairs $(a,b)\in K^2$, give the following $8$ distinct equivalence
classes
\begin{longtable}{|c|c|l|l|}
\hline
$d$ & $\varepsilon$ & Representative $(a,b)$ & $(a,b)$ of equivalent forms\\\hline
$1$ & $1$ & $(1,1)$ & \parbox{6cm}{$(u,u)$ \\
                                and $(p,p)$, $(up,up)$ for $p\equiv1\mod4$} \\\hline
$1$ & $-1$ & $(p,p)$ for $p \equiv 3\mod4$ & $(up,up)$ for $p\equiv3\mod4$\\\hline
$u$ & $1$ & $(1,u)$ & $(p,up)$ for $p\equiv3\mod4$\\\hline
$u$ & $-1$ & $(p,up)$ for $p\equiv 1\mod4$&\\\hline
$p$ & $1$ & $(1,p)$ &\\\hline
$p$ & $-1$ & $(u,up)$ &\\\hline
$up$ & $1$ & $(1,up)$ &\\\hline
$up$ & $-1$ & $(u,p)$ &\\\hline
\end{longtable}

Considering scaling, we can always make $a=1$, without loss of
generality. Thus, there are four different equivalence classes of
quadratic forms on $\QQ_p^2$ up to equivalence for every odd prime
$p$ \beq
  \label{eq:labelsino}
  x^2+y^2,\quad x^2+uy^2,\quad x^2+py^2,\quad ux^2+py^2
\eeq (it will become clear in a moment why we choose the last form
as written, and not as $x^2+upy^2$ or $x^2+p/uy^2$).

To determine which of them represents zero, we need to check whether $d\simeq -1$ in
$K$. %Theorem \ref{teor:Qrepns0}.
First, $-1\in(\bQ_p^\ast)^2$ for $p\equiv1\mod 4$, so $-1\simeq 1\in K$;
on the other hand, $-1\simeq u\in K$ for $p\equiv3\mod 4$ because $-1\notin(\bQ_p^\ast)^2$.
Now, $x^2+y^2$ has $d=1$: it represents $0$ for $p\equiv1\mod4$, but not for $p\equiv3\mod4$.
$x^2+uy^2$ has $d=u$: it represents $0$ for $p\equiv3\mod4$, but not for $p\equiv1\mod4$.
The last two forms of \eqref{eq:labelsino} do not represent $0$.

Proceeding similarly for $p=2$, we get the following $16$ classes of
quadratic forms on $\bQ_2^2$
\begin{longtable}{|c|c|l|l|}
\hline
$d$ & $\varepsilon$ & Representative $(a,b)$ & $(a,b)$ of equivalent forms\\\hline
$1$ & $1$ & $(1,1)$ &$(2,2),\,(5,5),\,(10,10)$\\\hline
$1$ & $-1$ & $(-1,-1)$& $(-2,-2),\,(-5,-5),\,(-10,-10)$\\\hline
$-1$ & $1$ & $(1,-1)$ & $(2,-2),\,(5,-5),\,(10,-10)$\\\hline
$-1$ & $-1$ & N/A${}^\dagger$ &N/A\\\hline
$2$ & $1$ & $(1,2)$ &$(-5,-10)$\\\hline
$2$ & $-1$ & $(-1,-2)$ &$(5,10)$\\\hline
$-2$ & $1$ & $(1,-2)$ &$(-1,2)$\\\hline
$-2$ & $-1$ & $(5,-10)$ &$(-5,10)$\\\hline
$5$ & $1$ & $(1,5)$ &$(-2,-10)$\\\hline
$5$ & $-1$ & $(-1,-5)$ &$(2,10)$\\\hline
$-5$ & $1$ & $(1,-5)$ &$(-1,5)$\\\hline
$-5$ & $-1$ & $(2,-10)$ &$(-2,10)$\\\hline
$10$ & $1$ & $(-2,-5)$ &$(1,-10)$\\\hline
$10$ & $-1$ & $(2,5)$ &$(-1,-10)$\\\hline
$-10$ & $1$ & $(1,-10)$ &$(-1,10)$\\\hline
$-10$ & $-1$ & $(2,-5)$ &$(-2,5)$\\\hline
\end{longtable}
\noindent
${}^\dagger${\small No $2$-adic quadratic form in dimension two can have determinant
$-1$ together with Hasse invariant $-1$.}

%\medskip
They reduce to the following $8$ classes up to scaling
\begin{alignat}{4}
  &x^2+y^2,  &\quad& x^2-y^2,  &\quad& x^2+2y^2,  &\quad& x^2-2y^2, \nonumber\\
  &x^2+5y^2, &\quad& x^2-5y^2, &\quad& x^2+10y^2, &\quad& x^2-10y^2.\nonumber
\end{alignat}
Only the second one represents $0$. In summary, we get the following classification.

\begin{prop}
\label{prop:quadform2} The quadratic forms on $\QQ_p^2$ that do not
represent $0$, are, up to equivalence
\begin{itemize}
  \item the following $3$ for every odd prime $p$
    \beq\notag
    Q_{-v}(\boldsymbol{x}) =x^2-vy^2,\quad
    Q_{p}(\boldsymbol{x})    =x^2+py^2,\quad
    Q_{up}(\boldsymbol{x}) =ux^2+py^2,
    \eeq
    where $v$ is as in Eq. \eqref{eq:woo};

  \item the following $7$ for $p=2$
    \begin{alignat}{2}
      Q_1(\boldsymbol{x})       &=x^2+y^2,\quad&
      Q_{\pm2}(\boldsymbol{x})  &=x^2\pm2y^2,\nonumber\\
      Q_{\pm5}(\boldsymbol{x})  &=x^2\pm5y^2,\quad&
      Q_{\pm10}(\boldsymbol{x}) &=x^2\pm10y^2.\nonumber
    \end{alignat}
\end{itemize}
The subscript of these quadratic forms denotes their determinant.
{\hfill\qed}
\end{prop}

When $p\equiv1\mod4$, the three definite forms are the restrictions of the
definite three-dimensional one to the planes $xy$, $xz$ and $yz$.
This is not true for $p\equiv3\mod4$, where $Q_+(\boldsymbol{x})$
reduces to $Q_{-v}(\boldsymbol{x})$ for $z=0$ and to the same $Q_p(\boldsymbol{x})$ for
both $x=0$ and $y=0$.
Similarly, the restriction of the definite three-dimensional form for $p=2$
to any reference plane $xy$, $xz$ and $yz$ gives only $Q_1(\boldsymbol{x})$ among the seven
possible forms in dimension two.

In contrast to the unique special orthogonal group $SO(2)_\RR$ on the real plane, there are
three (compact) special orthogonal groups on $\QQ_p^2$ for odd prime $p$, and seven for $p=2$,
up to isomorphisms, according to the last proposition. We call them $SO(2)_p^\kappa$ with
$\kappa$ denoting the determinant of the preserved quadratic form: these groups are
\beq
  \label{eq:so2pkappadef}
  SO(2)_p^\kappa=\{L\in\cM_{2\times2}(\QQ_p)\,:\,A_\kappa=L^\top A_\kappa L,\ \det L=1\},
\eeq endowed with the common matrix product. $A_\kappa$ denotes the
matrix representation of the quadratic form $Q_\kappa$ in the
canonical basis \beq\notag A_{-v}=\begin{pmatrix} 1&0\\0&-v
\end{pmatrix},\quad A_p=\begin{pmatrix} 1&0\\0&p \end{pmatrix},\quad
A_{up}=\begin{pmatrix} u&0\\0&p \end{pmatrix}, \eeq \beq\notag
A_1=\begin{pmatrix} 1&0\\0&1 \end{pmatrix},\quad
A_{\pm2}=\begin{pmatrix} 1&0\\0&\pm2 \end{pmatrix},\quad
A_{\pm5}=\begin{pmatrix} 1&0\\0&\pm5 \end{pmatrix},\quad
A_{\pm10}=\begin{pmatrix} 1&0\\0&\pm10 \end{pmatrix}. \eeq

%\medskip
We now proceed to the parametrisation of the two-dimensional
rotation groups. Consider the transformation of a vector in
$\QQ_p^2$ by some $\cR_\kappa\in SO(2)_p^\kappa$ \beq\notag
  \begin{pmatrix} x'\\y' \end{pmatrix}
  =
  \begin{pmatrix} a&b\\c&d \end{pmatrix}
  \begin{pmatrix} x\\y \end{pmatrix}.
\eeq
The orthogonality condition for $\cR_\kappa$ yields the system of equations
\beq
\label{eq:systso23}
\left\{\begin{aligned} &a^2+\alpha_\kappa c^2=1,\\
                       &b^2+\alpha_\kappa d^2=\alpha_\kappa,\\
                       &ab+\alpha_\kappa cd=0,
       \end{aligned}\right.
\eeq
where $\alpha_\kappa=-v,p,\frac{p}{u},1,\pm2,\pm5,\pm10$, respectively for the
quadratic forms with $\kappa=-v,p,up,\,p\neq2$ and $\kappa=1,\pm2,\pm5,\pm10,\,p=2$.
We parametrize the solutions of this system of three equations in four $p$-adic
unknowns through
\beq
\label{eq:paramsigmachiarot2}
\sigma=\pm\frac{c}{1+a}\in\QQ_p
\eeq
for $a\in\bZ_p\backslash\{-1\}$, and treat the case $a=-1$ separately.
The first equation of the system \eqref{eq:systso23} gives
\beq\notag
\alpha_\kappa=\frac{1-a^2}{c^2}
\eeq
when $c\neq0$. Manipulating Eq. \eqref{eq:paramsigmachiarot2},
we get $\frac{1+a}{c^2}=\frac{1}{(1+a)\sigma^2}$, $\sigma\neq0$, so
\beq
\label{eq:foralphaa2}
\alpha_\kappa=\frac{1-a}{(1+a)\sigma^2}\Rightarrow(1+\alpha_\kappa\sigma^2)a=1-\alpha_\kappa\sigma^2.
\eeq
The parameter $-\alpha_\kappa$ is not a square for any $\kappa$ of the considered
quadratic forms (because none of them represents zero).
This means that $1+\alpha_\kappa\sigma^2\neq0$ for all $\alpha_\kappa$ and $\sigma\in\QQ_p$.
%because when $\sigma\neq 0$ then $1+\alpha_\kappa\sigma^2=0\Leftrightarrow -\alpha_\kappa=\sigma^{-2}$ has no solution since $-\alpha_\kappa$ is not a square.
Now Eq. \eqref{eq:foralphaa2} gives
\beq\notag
a=\frac{1-\alpha_\kappa\sigma^2}{1+\alpha_\kappa\sigma^2}.
\eeq
The first equation of \eqref{eq:systso23} now provides
\beq\notag
  c^2 = \frac{1-a^2}{\alpha_\kappa}=\frac{4\sigma^2}{\left(1+\alpha_\kappa\sigma^2\right)^2}
  \Rightarrow
  c = \pm_c \frac{2\sigma}{1+\alpha_\kappa\sigma^2}.
\eeq
From the third equation of \eqref{eq:systso23}, when $a\neq0$, we get
\beq
\label{eq:anot0}
b^2=\alpha_\kappa\left(\frac{1}{a^2}-1\right)d^2
\eeq
which plugged into the second one, gives $d^2=a^2$, hence $d=\pm_da$.

Now the second equation of \eqref{eq:systso23} gives
\beq\notag
b^2=\alpha_\kappa^2c^2 \Rightarrow b=\pm_b\frac{2\alpha_\kappa\sigma}{1+\alpha_\kappa\sigma^2}.
\eeq
We arrive at
\beq
\label{eq:detmatpersegnisco}
\cR_\kappa(\sigma) =
\begin{pmatrix}
  a(\sigma)&b(\sigma)\\
  c(\sigma)&d(\sigma)
\end{pmatrix} =
\begin{pmatrix}
   \phantom{\pm_c}\frac{1-\alpha_\kappa\sigma^2}{1+\alpha_\kappa\sigma^2}
     &\pm_b \frac{2\alpha_\kappa\sigma}{1+\alpha_\kappa\sigma^2}\\
   \pm_c \frac{2\sigma}{1+\alpha_\kappa\sigma^2}
     &\pm_d\frac{1-\alpha_\kappa\sigma^2}{1+\alpha_\kappa\sigma^2}
\end{pmatrix},
\eeq where the signs $\pm_b,\pm_c,\pm_d$ are a priori unrelated. The
determinant of the matrix \eqref{eq:detmatpersegnisco} is \beq\notag
\pm_d\left(\frac{1-\alpha_\kappa\sigma^2}{1+\alpha_\kappa\sigma^2}\right)^2
-\pm_b\pm_c\alpha_\kappa\left(\frac{2\sigma}{1+\alpha_\kappa\sigma^2}\right)^2=\frac{\pm_d1-2\alpha_\kappa\sigma^2(\pm_d1+\pm_b\pm_c2)\pm_d\alpha_\kappa^2\sigma^4}{1+2\alpha
_\kappa\sigma^2+\alpha_\kappa^2\sigma^4}, \eeq and it must be $1$.
This happens when $\pm_d1=1$ and $\pm_d1+\pm_b\pm_c2=-1$, i.e.,
$\pm_c=\mp_b$. Therefore, \beq\label{eq:rotaz2param2x2}
\cR_\kappa(\sigma)=
\begin{pmatrix}
  \phantom{\pm}\frac{1-\alpha_\kappa\sigma^2}{1+\alpha_\kappa\sigma^2}
     &\mp \frac{2\alpha_\kappa\sigma}{1+\alpha_\kappa\sigma^2}\\
  \pm \frac{2\sigma}{1+\alpha_\kappa\sigma^2}
     &\phantom{\pm}\frac{1-\alpha_\kappa\sigma^2}{1+\alpha_\kappa\sigma^2}
\end{pmatrix},
\eeq
with linked signs.

In the derivation, at Eq. \eqref{eq:anot0}, we had to assume $a\neq0$.
When $a=0$, we have $b^2=\alpha_\kappa$, $c^2=1/\alpha_\kappa$, $d=0$,
which provides a solution when $\alpha_\kappa$ is a square. In this case,
$a=0$ corresponds to $\sigma=\pm c$, included in Eq. \eqref{eq:rotaz2param2x2}.

The case $c=0$ left out in the system of equations gives $\cR_\kappa=\pm I$,
using that the determinant is $1$. Thus, $c=0$ is equivalent to $\sigma=0$ when
$a\neq-1$ in Eq. \eqref{eq:paramsigmachiarot2}, leading to $\cR_\kappa(0)= I$.
This leaves one solution of Eq. \eqref{eq:rotaz2param2x2} unaccounted for,
which is $\cR_\kappa=- I$, reached from the system of equations with $a=-1$,
for which $\sigma$ is not well-defined in $\QQ_p$. But this sign change from
the identity matrix can be included in Eq. \eqref{eq:rotaz2param2x2} by considering
$\sigma\in\QQ_p\cup\{\infty\}$.

\begin{rem}
\label{oss:inftyinclusion} Note that $\lim\limits_{n\to\infty}
p^n=0$ in $\bQ_p$ with respect to the $p$-adic norm, since $|p^n|_p$ converges to $0$ by increasing $n$. Then, $\infty$ on the
$p$-adic field is formally introduced as the limit
$\lim\limits_{n\to\infty}p^{-n}$ (the inverse of $0$). Indeed the
matrix entries in Eq. \eqref{eq:rotaz2param2x2} are such that \beq\notag
\begin{aligned}
&\lim_{n\to\infty}\frac{1-\alpha_\kappa p^{-2n}}{1+\alpha_\kappa p^{-2n}}=
\lim_{n\to\infty}\frac{p^{-2n}(p^{2n}-\alpha_\kappa)}{p^{-2n}(p^{2n}+\alpha_\kappa)}=-1,\\
&\lim_{n\to\infty}\frac{2p^{-n}}{1+\alpha_\kappa p^{-2n}} = \lim_{n\to\infty}\frac{p^{-2n}(2p^{n})}{p^{-2n}(p^{2n}+\alpha_\kappa)}=0
\end{aligned}
\eeq
yielding $\cR_\kappa(\infty)=-I$.
\end{rem}

As a parameter $\sigma$ has a sign that is linked to its inverse
(as we are going to see in a moment): it can swap between $+$ and $-$ variations of
$b$ and $c$ even when we fix one choice of $\pm_c$. Hence we can keep only
one sign in Eq. \eqref{eq:rotaz2param2x2} without loss of generality.
We have thus proved the following.

\begin{theor}
\label{thm:SO2} A rotation of $SO(2)_p^\kappa$ takes the following
matrix form in the canonical basis of $\QQ_p^2$ \beq
\label{eq:rotazgeneric2} \cR_\kappa(\sigma)=
\begin{pmatrix}
  \frac{1-\alpha_\kappa\sigma^2}{1+\alpha_\kappa\sigma^2}
    & -\frac{2\alpha_\kappa\sigma}{1+\alpha_\kappa\sigma^2}\\
  \frac{2\sigma}{1+\alpha_\kappa\sigma^2}
    & \phantom{-}\frac{1-\alpha_\kappa\sigma^2}{1+\alpha_\kappa\sigma^2}
\end{pmatrix},
\eeq
with $\sigma\in\QQ_p\cup\{\infty\}$, $\alpha_\kappa\in\{-v,p,\frac{p}{u}\}$ and $\alpha_k\in\{1,\pm2,\pm5,\pm10\}$ respectively for $\kappa=-v,p,up$ ($p$ odd)
and $\kappa=1,\pm2,\pm5,\pm10$ ($p=2$).

In addition, this parametrization is one-to-one, i.e., it results in
different special orthogonal transformations for different $\sigma$.
{\hfill\qed}
\end{theor}

\begin{rem}
\label{oss:swapsign} For every $\kappa$ of the definite rank-$2$
quadratic forms \beq \label{eq:switchsign}
\cR_\kappa\left(-\frac{1}{\alpha_\kappa\sigma}\right)=-\cR_\kappa\left(\sigma\right),
\ \  \sigma\in\bQ_p\cup\{\infty\}, \eeq because if we replace
$\sigma\mapsto -\frac{1}{\alpha_\kappa\sigma}$ in the matrices
\eqref{eq:rotazgeneric2}, then the matrix elements change as follows
\begin{align*}
  \frac{1-\alpha_\kappa\sigma^2}{1+\alpha_\kappa\sigma^2}
     &\mapsto\frac{1-\alpha_\kappa(\alpha_\kappa\sigma)^{-2}}{1+\alpha_\kappa(\alpha_\kappa\sigma)^{-2}}
       =\frac{\alpha_\kappa\sigma^2-1}{\alpha_\kappa\sigma^2+1}
       =-\frac{1-\alpha_\kappa\sigma^2}{1+\alpha_\kappa\sigma^2},\nonumber\\
  -\frac{2\alpha_\kappa\sigma}{1+\alpha_\kappa\sigma^2}
     &\mapsto-\frac{2\alpha_\kappa(-\alpha_\kappa\sigma)^{-1}}{1+\alpha_\kappa(\alpha_\kappa\sigma)^{-2}}
       =\frac{2\alpha_\kappa\sigma}{1+\alpha_\kappa\sigma^2},\\
  \frac{2\sigma}{1+\alpha_\kappa\sigma^2}
     &\mapsto \frac{2(-\alpha_\kappa\sigma)^{-1}}{1+\alpha_\kappa(\alpha\sigma)^{-2}}
       =-\frac{2\sigma}{1+\alpha_\kappa\sigma^2}.\nonumber
\end{align*}
We need $\sigma\in\bQ_p\cup\{\infty\}$ in order for this transformation to be
well-defined for $\sigma=0$, too, which is guaranteed by
\beq\notag
\cR_\kappa\left(\infty\right)=-\cR_\kappa\left(0\right)=- I.
\eeq
\end{rem}

\begin{rem}
\label{oss:perriduzmodpk}
As a corollary of the parameterization \eqref{eq:rotazgeneric2}, we can see
directly that the matrix entries
$\frac{1-\alpha_\kappa\sigma^2}{1+\alpha_\kappa\sigma^2}$,
$\frac{2\sigma}{1+\alpha_\kappa\sigma^2}$ and
$-\frac{2\alpha_\kappa\sigma}{1+\alpha_\kappa\sigma^2}$ of the rotations
are $p$-adic integers for every $\alpha_\kappa$, every $\sigma\in\QQ_p\cup\{\infty\}$
and every prime $p$,
in accordance with Theorem \ref{thm:compactness}.
This can be easily checked for parameters $\sigma\in\ZZ_p$:
\begin{itemize}
\item $1+\alpha_\kappa\sigma^2\not\equiv0\mod p$, for every $\kappa,\,\sigma\in\bZ_p,\,p\neq2$, for $\kappa=\pm2,\pm10,\,\sigma\in\bZ_2$, and for $\kappa=1,\pm5,\,\sigma\in2\bZ_2$. In these cases $(1+\alpha_\kappa\sigma^2)^{-1}\in\bU$, which multiplied with $1-\alpha_\kappa\sigma^2,\,2\sigma,\,-2\alpha_\kappa\sigma\in\bZ_p$ give $p$-adic integer matrix entries;
\item If $\sigma\in\bZ_2$, $\sigma\equiv1\mod2$, then $1+\sigma^2,\,1\pm5\sigma^2\equiv0\mod2$: they are invertible in $\bQ_2$ but not in $\bZ_2$. Here the calculus modulo $2$ is not enough to verify that the associated matrix entries are $2$-adic integers. But writing $\sigma=1+2\sigma',\,\sigma'\in\bZ_2$ we get     \beq\notag\begin{aligned}
    &\frac{1-\sigma^2}{1+\sigma^2}=\frac{-4(\sigma'+{\sigma'}^2)}{2\big(1+2(\sigma'+{\sigma'}^2)\big)}=\frac{-2(\sigma'+{\sigma'}^2)}{1+2(\sigma'+{\sigma'}^2)},\\
    &\frac{2\sigma}{1+\sigma^2}=\frac{2(1+2\sigma')}{2\big(1+2(\sigma'+{\sigma'}^2)\big)}=\frac{1+2\sigma'}{1+2(\sigma'+{\sigma'}^2)}.
\end{aligned}\eeq
$1+2(\sigma'+{\sigma'}^2)\not\equiv0\mod2$ is invertible in $\bZ_2$.
\end{itemize}
We have shown that the entries of the matrices of $SO(2)_p^\kappa$ are $p$-adic integers when $\sigma\in\bZ_p$, for every $\kappa$ and prime $p$. But then we can distinguish two branches for $\sigma\in\bQ_p\cup\{\infty\}$: either $\sigma\in\bZ_p$ or $\sigma^{-1}\in p\bZ_p$ (including $\infty$ formally when $\sigma=0$). Since we want to exploit Eq. \eqref{eq:switchsign} for a non integer parameter, either $\sigma\in\bZ_p$ or:
\begin{itemize}
    \item $\sigma=-\frac{1}{\alpha_\kappa\tau},\ \tau\in p\bZ_p$, for $p\neq2,\kappa=-v$ and for $p=2,\kappa=1,\pm5$,
    \item $\sigma=-\frac{1}{\alpha_\kappa\tau},\ \tau\in \bZ_p$, for $p\neq2,\kappa=p,up$ and for $p=2,\kappa=\pm2,\pm10$.
\end{itemize}
In each of the cases of the first point
\beq\notag\begin{aligned}
 SO(2)_p^\kappa &=\left\{\cR_\kappa(\sigma)\,:\,\sigma\in\bQ_p\cup\{\infty\}\right\} \\
   &=\left\{\cR_\kappa(\sigma)\,:\,\sigma\in\bZ_p\right\}
      \cup \left\{\cR_\kappa\left(-\frac{1}{\alpha_\kappa\tau}\right)\,:\,\tau\in p\bZ_p\right\}\\
   &=\left\{\cR_\kappa(\sigma)\,:\,\sigma\in\bZ_p\right\}
      \cup \left\{-\cR_\kappa\left(\tau\right)\,:\,\tau\in p\bZ_p\right\}.
\end{aligned}\eeq
The second set includes $\cR_\kappa(\infty)$ for $\tau=0$.

Similarly for all the cases in the second point
\beq\notag%\begin{aligned}
 SO(2)_p^\kappa
   =\left\{\cR_\kappa(\sigma)\,:\,\sigma\in\bQ_p\cup\{\infty\}\right\}
   =\left\{\cR_\kappa(\sigma)\,:\,\sigma\in\bZ_p\right\}
     \cup \left\{-\cR_\kappa\left(\tau\right)\,:\,\tau\in \bZ_p\right\}.
%\end{aligned}
\eeq
In this way, the parameterisation of the groups $SO(2)_p^\kappa$ is in terms
of $p$-adic integers only. The second branch for the parameter $\sigma$ gives
the same matrix entries of the first one up to a sign: these entries are $p$-adic
integers as shown for the first branch.

This is useful when projecting $SO(3)_p$ modulo $p^k$, to be able to enumerate
the elements of its subgroups of rotations around the reference axes \cite{Ilaria:tesi}.
\end{rem}

%\medskip
\textbf{Interlude: Euclidean geometrical interpretation.}
The introduction of a parameter $\sigma$ leading to $a=\frac{1-\alpha\sigma^2}{1+\alpha\sigma^2}$
in Eq. \eqref{eq:rotazgeneric2} is inspired by Euclidean geometry.
Indeed, recalling the \emph{tangent half-angle formulae}
\beq\notag
  \cos\theta=\frac{1-\tan^2(\theta/2)}{1+\tan^2(\theta/2)},
  \quad
  \sin\theta=\frac{2\tan(\theta/2)}{1+\tan^2(\theta/2)},
\eeq
and letting
\beq
\label{eq:condINTsigma}
\alpha_\kappa=1, \quad \sigma=\tan\left(\frac{\theta}{2}\right)\in\RR,
\eeq
  Eq. \eqref{eq:rotazgeneric2} takes the form
\beq\notag
\cR(\theta) =
\begin{pmatrix}
  \cos\theta&-\sin\theta\\\sin\theta&\cos\theta
\end{pmatrix},
\eeq
which is the matrix for a rotation by an angle $\theta$ in the real plane. The
special case $-I$ is given by $\theta=\pi$, for which $\tan(\theta/2)$ diverges.
It is as if we have been treating rotations in $\QQ_p^2$ in terms of the tangent
of the rotation angle rather than the angle itself. It is worth noticing that
the usage of tangent of angles (in place of angles) in trigonometry dates back
to ancient Babylon: the Babylonian table, \emph{Plimpton 322} (written about 1800 BC),
may be interpreted as a first trigonometric table, where Pythagorean triples were
written as ratios between the sides of right triangles \cite{Cowen}.

This geometrical interpretation motivates our notation $\cR_\kappa(\sigma)$ for
a rotation of parameter $\sigma$.

%\medskip
\textbf{Composition laws of rotations.}
The composition of two rotations of $SO(2)_p^\kappa$, for any fixed $\kappa$,
is given by the product of their matrix forms \eqref{eq:rotazgeneric2}. It turns out to
take the very simple form
\beq
  \label{eq:compSamxi2}
  \cR_\kappa(\sigma)\cR_\kappa(\tau)=\cR_\kappa\left(\frac{\sigma+\tau}{1-\alpha_\kappa\sigma\tau}\right),
\eeq
for every $\sigma,\tau\in\bQ_p\cup\{\infty\}$. For the infinitely many values of
$\sigma,\tau\in\bQ_p$ such that $1-\alpha_\kappa\sigma\tau=0$, the right-hand-side
is $\cR_\kappa(\infty) = -I$, according to Remark \ref{oss:inftyinclusion}.
If $\tau\in\bQ_p$ while $\sigma$ is $\infty$ (or vice versa), the argument on
the right-hand side is
\beq\notag
  \lim_{n\to\infty}\frac{p^{-n}+\tau}{1-\alpha_\kappa p^{-n}\tau}
   =\lim_{n\to\infty}\frac{p^{-n}(1+p^n\tau)}{p^{-n}(p^n-\alpha_\kappa\tau)}
   =-\frac{1}{\alpha_\kappa\tau},
\eeq
so Eq. \eqref{eq:compSamxi2} becomes
$-\cR_\kappa(\tau)=\cR_\kappa\left(-\frac{1}{\alpha_\kappa\tau}\right)$ in agreement
with Eq. \eqref{eq:switchsign}. If both $\sigma$ and $\tau$ are $\infty$, the above
composition gives correctly the identity, since the rotation parameter on the right is
\beq\notag
\lim_{n,m\to\infty}\frac{p^{-n}+p^{-m}}{1-\alpha_\kappa p^{-n-m}}
  =\lim_{n,m\to\infty}\frac{p^{-n-m}(p^m+p^n)}{p^{-n-m}(p^{n+m}-\alpha_\kappa)}
  =0.
\eeq

We recall that
\beq\notag
\tan(\alpha+\beta)=\frac{\tan\alpha+\tan\beta}{1-\tan\alpha\tan\beta}
\eeq
which is equivalent to the arguments of Eq. \eqref{eq:compSamxi2} if we use
Eq. \eqref{eq:condINTsigma}.

\begin{prop}
\label{prop:subAXISgroup2}
$SO(2)_p^\kappa$ defined in Eq. \eqref{eq:so2pkappadef} is an abelian group,
for every $\kappa$ and prime $p$.
\end{prop}
\begin{proof}
This is straightforward to prove via Eq. \eqref{eq:compSamxi2}, and it is in
analogy with $SO(2)_{\RR}$, which is isomorphic to the group of addition
or real angles $\theta$ modulo $2\pi$, but reinterpreted through the half-angle
tangent.
Furthermore, $\cR_\kappa(\sigma)^{-1}=\cR_\kappa(-\sigma)$
for all $\sigma\in\bQ_p$ and $\cR_\kappa(\infty)^{-1}=\cR_\kappa(\infty)$.
\end{proof}

%\medskip
\textbf{Rotations parametrized by the $p$-adic projective line.}
The form of the parametrization \eqref{eq:rotazgeneric2} of two-dimensional rotations,
and their composition law \eqref{eq:compSamxi2}, and in particular the necessity to
add a point at infinity to the $p$-adic affine line, suggest a description
in terms of projective geometry.

As usual, we identify the projective line $P^1(\QQ_p)$ with the
equivalence classes $[s:t] = \QQ_p^*(s,t) \subset \QQ_p^2\setminus(0,0)$
of nonzero vectors under multiplication by (nonzero) scalars. Representatives are $[s:1]$ and $\infty \equiv [1:0]$, allowing us to think of
$P^1(\QQ_p) = \QQ_p \cup \{\infty\}$ as the $p$-adic line closed with
a point at infinity. The parameter in Eq. \eqref{eq:rotazgeneric2}
is $\sigma = s/t$, which is well-defined on each equivalence class.
We can write the rotations as
bona fide functions of the points $[s:t]$ of $P^1(\QQ_p)$
\beq
\label{eq:rotazgeneric2-proiett}
\cR_\kappa([s:t])=
\begin{pmatrix}
  \frac{t^2-\alpha_\kappa s^2}{t^2+\alpha_\kappa s^2}
    & -\frac{2\alpha_\kappa st}{t^2+\alpha_\kappa s^2} \\
  \frac{2 st}{t^2+\alpha_\kappa s^2}
    & \phantom{-}\frac{t^2-\alpha_\kappa s^2}{t^2+\alpha_\kappa s^2}
\end{pmatrix},
\eeq
where $\alpha_\kappa\in\{-v,p,\frac{p}{u}\}$ and $\alpha_k\in\{1,\pm2,\pm5,\pm10\}$
respectively for $\kappa=-v,p,up$ ($p$ odd) and $\kappa=1,\pm2,\pm5,\pm10$ ($p=2$).
Furthermore, the uniqueness of the parametrization \eqref{eq:rotazgeneric2-proiett}
as a function of $[s:t] \in P^1(\QQ_p)$ is inherited from Theorem \ref{thm:SO2}.

Looking at $\sigma$ as a ratio reveals the underlying structure, too, of the
composition law \eqref{eq:compSamxi2}, which in projective line coordinates
becomes
\beq
  \label{eq:compSamxi2-proiett}
  \cR_\kappa([s:t])\cR_\kappa([u:v])
      =\cR_\kappa\left([sv+tu:tv-\alpha_\kappa su]\right).
\eeq
Note that for each fixed $[s:t]$, this is a projective linear transformation
of $[u:v]$ (and vice versa). Indeed, the mapping
\begin{align*}
  T_\kappa: P^1(\QQ_p) &\longrightarrow PGL(2,\QQ_p) \nonumber\\
            [s:t]      &\longmapsto T_\kappa([s:t]) := \QQ_p^*
                                          \begin{pmatrix} t & -\alpha_\kappa s \\ s & t \end{pmatrix}
\end{align*}
defines abelian subgroups $\cT_\kappa := T_\kappa(P^1(\QQ_p))
\subset PGL(2,\QQ_p)$ (the invertible $2\times 2$-matrices over
$\QQ_p$ modulo scalars $\QQ_p^*$), each a rational image of the
projective line. $T_\kappa$ translates the composition law Eq.
\eqref{eq:compSamxi2-proiett} into a simple matrix multiplication
\begin{equation}\notag
  T_\kappa([s:t]) T_\kappa([u:v]) = T_\kappa([sv+tu:tv-\alpha_\kappa su]),
\end{equation}
This means that we can identify $SO(2)_p^\kappa$ with the subgroup
$\cT_\kappa \subset PGL(2,\QQ_p)$, $T_\kappa$ providing the isomorphism. If, by abuse of notation, we refer to $T_\kappa$ as the map from $P^1(\QQ_p)$ to its image $\cT_\kappa$, the isomorphic map from $\cT_\kappa$ to $SO(2)_p^\kappa$ is explicitly given by first applying $T_\kappa^{-1}:\QQ_p^*\begin{pmatrix} t & -\alpha_\kappa s \\ s & t \end{pmatrix}\mapsto [s,t]$ and then parameterizing an element of $SO(2)_p^\kappa$ as in Eq. \eqref{eq:rotazgeneric2-proiett}.

\begin{rem}\label{rem:omeopar}
Consider the quotient topologies on $P^1(\mathbb{Q}_p)$ and $PGL(2,\QQ_p)$, where $\mathbb{Q}_p^2\backslash(0,0)$ and $GL(2,\mathbb{Q}_p)$ are naturally endowed with their $p$-adic topology, and consider the topology of the $T_\kappa$ as subspaces of $PGL(2,\QQ_p)$. The map $T_\kappa$, from $P^1(\QQ_p)$ to its image $\cT_\kappa$, is clearly a homeomorphism, for every $\kappa$: it is continuous, as well as its inverse $T_\kappa^{-1}$, as it can be directly seen from their explicit expressions above. Now, let $SO(2)_p^\kappa$ be supplied with its natural $p$-adic topology. The parameterization Eq. \eqref{eq:rotazgeneric2-proiett} is continuous in $s$ and $t$, as its components are well defined rational functions. Hence it is continuous in $[s:t]$, since the parameterization does not depend on the representatives of the equivalence classes $[s:t]$. The inverse map, providing $[s:t]$ from such a parameterized matrix $\cR_\kappa([s:t])$, is given by 
\begin{equation}\notag
 \begin{pmatrix}
   a&b\\c&d
 \end{pmatrix}\coloneqq\begin{pmatrix}
  \frac{t^2-\alpha_\kappa s^2}{t^2+\alpha_\kappa s^2}
    & -\frac{2\alpha_\kappa st}{t^2+\alpha_\kappa s^2} \\
  \frac{2 st}{t^2+\alpha_\kappa s^2}
    & \phantom{-}\frac{t^2-\alpha_\kappa s^2}{t^2+\alpha_\kappa s^2}
\end{pmatrix}\mapsto [c:1+a]=\left[s\frac{2t}{t^2+\alpha_\kappa s^2}:t\frac{2t}{t^2+\alpha_\kappa s^2}\right]=[s:t]
\end{equation}
whenever $a\neq-1$, otherwise
\begin{equation}\notag
 \begin{pmatrix}
   a&b\\c&d
 \end{pmatrix}=\begin{pmatrix}
   -1&0\\0&-1
 \end{pmatrix}\mapsto [1:0].
 \end{equation}
This inverse map is again continuous, thus Eq. \eqref{eq:rotazgeneric2-proiett} provides a homeomorphism between $P^1(\mathbb{Q}_p)$ and $SO(2)_p^\kappa$, for every $\kappa$. It immediately follows that the parameterization \eqref{eq:rotazgeneric2} is a homeomorphism. We have also shown that the isomorphic map between $\cT_\kappa$ and $SO(2)_p^\kappa$ is a homeomorphism, for every $\kappa$.
\end{rem}

\section{Parametrization of rotations in $SO(3)_p$}
\label{gendir} Returning to $p$-adic three-space $\QQ_p^3$, we found
that every element of $SO(3)_p$ has a fixed axis of rotation. So,
now we consider a rotation around a general direction
$\boldsymbol{n}\in\bQ_p^3\setminus\boldsymbol{0}$. Let
$\boldsymbol{g},\boldsymbol{h}\in\bQ_p^3$ be vectors spanning
$V=\boldsymbol{n}^\perp$. Restricting our attention to $V$ means
reducing a vector $\boldsymbol{s}\in\bQ_p^3$, with coordinates
$(s_1,s_2,s_3)$ with respect to the basis
$(\boldsymbol{g},\boldsymbol{h},\boldsymbol{n})$, to a vector
$\boldsymbol{s}_{\left|V\right.}$ with coordinates $(s_1,s_2,0)$. If
$\boldsymbol{g}=(g_1,g_2,g_3)$, $\boldsymbol{h}=(h_1,h_2,h_3)$ and
$\boldsymbol{n}=(n_1,n_2,n_3)$ with respect to the canonical basis
$(\boldsymbol{e}_1,\,\boldsymbol{e}_2,\,\boldsymbol{e}_3)$ of
$\bQ_p^3$, then the canonical coordinates of $\boldsymbol{s}_{\vert
V}$ are given by \beq\notag
\begin{pmatrix}
w_1\\w_2\\w_3
\end{pmatrix} = \begin{pmatrix}
g_1&h_1&n_1\\g_2&h_2&n_2\\g_3&h_3&n_3
\end{pmatrix}\begin{pmatrix}
s_1\\s_2\\0
\end{pmatrix}
\eeq
or in short
\beq\notag
\boldsymbol{w}=M\boldsymbol{s}_{\left|V\right.}.
\eeq
The quadratic form preserved by $SO(3)_p$ is $Q_+(\boldsymbol{x})=\sum_{i=1}^3a_ix_i^2$
in the canonical basis where, $(a_i)=(1,-v,p)$ for $p\neq2$ and $(a_i)=(1,1,1)$ for $p=2$.
Its restriction to $V$ is
\begin{equation}\notag\begin{split}
Q_{+\vert V}(\boldsymbol{w}) &=\sum_{i=1}^3a_iw_i^2
                             =\sum_{i=1}^3a_i\left(g_is_1+h_is_2\right)^2\\
           &=\left(\sum_{i=1}^3a_ig_i^2\right)s_1^2
              +2\left(\sum_{i=1}^3a_ig_ih_i\right)s_1s_2
              +\left(\sum_{i=1}^3a_ih_i^2\right)s_2^2\\
           &=Q_+(\boldsymbol{g})s_1^2+2b(\boldsymbol{g},\boldsymbol{h}) s_1s_2+Q_+(\boldsymbol{h})s_2^2.
\end{split}\end{equation}
If $\boldsymbol{g}\perp\boldsymbol{h}$, then \beq \notag Q_{+\vert
V}(\boldsymbol{w}) = gs_1^2+hs_2^2 \eeq where $g=Q_+(\boldsymbol{q})$,
$h=Q_+(\boldsymbol{h})$.

We consider $\boldsymbol{t}\in\bQ_p^3$ with coordinates $(t_1,t_2,t_3)$ with respect to
$(\boldsymbol{g},\boldsymbol{h},\boldsymbol{n})$ such that
\beq\notag
\begin{pmatrix}
t_1\\t_2
\end{pmatrix} = \begin{pmatrix}
a&b\\c&d
\end{pmatrix}\begin{pmatrix}
s_1\\s_2
\end{pmatrix}
\eeq
and we want to find this matrix rotating $\boldsymbol{s}_{\left|V\right.}$ to
$\boldsymbol{t}_{\left|V\right.}$. If $\boldsymbol{u}:=M\boldsymbol{t}_{\left|V\right.}$, we must have
\beq\notag
Q_{+\vert V}(\boldsymbol{u}) = Q_{+\vert V}(\boldsymbol{w}).
\eeq
This implies
\beq\label{sysQ1}
\left\{\begin{aligned}
  ga^2+hc^2 &=g,\\
  gb^2+hd^2 &=h,\\
  gab+hcd   &=0,
\end{aligned}\right.
\eeq
which generalizes Eq. \eqref{eq:systso23} by setting
\beq\notag
\alpha:=\frac{h}{g}\in\bQ_p^\ast.
\eeq
Now we proceed as in the previous section, and introduce
\beq\notag
  \sigma=\pm\frac{c}{1+a}\in\bQ_p,
\eeq
leaving out the case $a\neq-1$ for the moment. This leads to
\beq\notag
  (1+\alpha\sigma^2)a=1-\alpha\sigma^2,
\eeq
and in order to extract $a$ we need the following results.

\begin{prop}
\label{prop:numerattt}
For any orthogonal basis $(\boldsymbol{g},\boldsymbol{h},\boldsymbol{n})$ of $\bQ_p^3$,
\beq\notag
Q_+(\boldsymbol{g})Q_+(\boldsymbol{h})Q_+(\boldsymbol{n})
  =\begin{cases} -vpt^2 &\text{ for } p \text{ odd}, \\
                    t^2 &\text{ for } p=2,
   \end{cases}
\eeq
for some $t\in\bQ_p^\ast$.
\end{prop}
\begin{proof}
We recall that the determinant is an invariant of quadratic forms in
$K=\bQ_p^\ast/(\bQ_p^\ast)^2$ under change of basis.
Given bases $\mathcal{B}$ and $\mathcal{C}$ of $\bQ_p^3$, there exists a unique
linear transformation $L:\bQ_p^3\rightarrow\bQ_p^3$ mapping $\mathcal{B}$ to $\mathcal{C}$.
If a quadratic form over $\bQ_p$ has matrix representation $M$ with respect to $\mathcal{B}$,
it has a matrix representation $M'=LML^\top $ with respect to $\mathcal{C}$.
Now, $\det LML^\top  =(\det L)^2(\det M)$ implies $\det M\simeq\det M'$ in $K$.
Our form $Q_+$ is represented in the canonical basis
$\mathcal{B}=(\boldsymbol{e}_1,\boldsymbol{e}_2,\boldsymbol{e}_3)$
of $\bQ_p^3$ by $A=\text{diag}(1,-v,p)$ for $p\neq2$ and by $A= I$ for $p=2$.
We have $\det A=Q_+(\boldsymbol{e}_1)Q_+(\boldsymbol{e}_2)Q_+(\boldsymbol{e}_3)$, which
is equal to $-vp$ when $p\neq2$ and to $1$ when $p=2$.
On the other hand, $Q_+$ has diagonal matrix form with respect to any orthogonal
basis $\mathcal{C}=(\boldsymbol{g},\boldsymbol{h},\boldsymbol{n})$ of $\bQ_p^3$,
where its determinant is $Q_+(\boldsymbol{g})Q_+(\boldsymbol{h})Q_+(\boldsymbol{n})$. Hence,
$Q_+(\boldsymbol{g})Q_+(\boldsymbol{h})Q_+(\boldsymbol{n})\simeq -vp$ in $K$ for $p\neq2$,
and $Q_+(\boldsymbol{g})Q_+(\boldsymbol{h})Q_+(\boldsymbol{n})\simeq 1$ for $p=2$.
\end{proof}

\begin{cor}
\label{cor:forueedfapaval}
$-\alpha\in\bQ_p^\ast$ is never a square.
\end{cor}
\begin{proof}
From Proposition \ref{prop:numerattt}, we have
\beq\notag
Q_+(\boldsymbol{g})Q_+(\boldsymbol{h})
 \simeq \begin{cases} -vpQ_+(\boldsymbol{n}) &\text{ if } p \text{ odd},\\
                         Q_+(\boldsymbol{n}) &\text{ if } p=2,
        \end{cases}
\eeq
which gives
\beq\notag
-\alpha=-\frac{Q_+(\boldsymbol{h})}{Q_+(\boldsymbol{g})}\simeq -Q_+(\boldsymbol{g})Q_+(\boldsymbol{h})\simeq\left\{\begin{aligned}
  vpQ_+(\boldsymbol{n}),&\ p\neq2,\\
   -Q_+(\boldsymbol{n}),&\ p=2.
\end{aligned}\right.
\eeq
Accordingly, $-\alpha$ is a square if and only if $Q_+^{(4)}(\boldsymbol{n},s)=0$ for some $s\in\bQ_p^\ast$.
However $Q_+^{(4)}$ does not represent $0$.
\end{proof}

This implies $1+\alpha\sigma^2\neq0$ for all $\alpha\in\bQ_p^\ast,\sigma\in\bQ_p$, and we get
\beq\notag
  a=\frac{1-\alpha\sigma^2}{1+\alpha\sigma^2},
\eeq and, therefore, \beq \label{eq:rotazparam2x2}
\cR_{\boldsymbol{n}}(\sigma)_{\vert\boldsymbol{n}^\perp}=\begin{pmatrix}
a(\sigma)&b(\sigma)\\c(\sigma)&a(\sigma)
\end{pmatrix}=\begin{pmatrix}
\frac{1-\alpha\sigma^2}{1+\alpha\sigma^2}&- \frac{2\alpha\sigma}{1+\alpha\sigma^2}\\  \frac{2\sigma}{1+\alpha\sigma^2}&\frac{1-\alpha\sigma^2}{1+\alpha\sigma^2}
\end{pmatrix}.
\eeq
We need to consider $\sigma\in\bQ_p\cup\{\infty\}$ as in Remark
\ref{oss:inftyinclusion} in order to include the solution to Eq. \eqref{sysQ1}
with $a=-1$, which is
\beq\notag
  \cR_{\boldsymbol{n}}(\infty)_{\vert\boldsymbol{n}^\perp} = -I.
\eeq

We extend the result of Eq. \eqref{eq:rotazparam2x2} to the whole of $\bQ_p^3$:
\beq\notag
\begin{pmatrix}
t_1\\t_2\\t_3
\end{pmatrix}=\begin{pmatrix}
a(\sigma)&b(\sigma)&0\\c(\sigma)&a(\sigma)&0\\0&0&1
\end{pmatrix}\begin{pmatrix}
s_1\\s_2\\s_3
\end{pmatrix}
\eeq
choosing $t_3=s_3$, or in short
\beq
\label{eq:relationrot}
\boldsymbol{t}=\cR_{\boldsymbol{n}}(\sigma)\boldsymbol{s}.
\eeq
This concludes the proof of the following theorem.

\begin{theor}
A rotation of $SO(3)_p$ around $\boldsymbol{n}$ takes the following matrix form with respect to an orthogonal basis $(\boldsymbol{g},\boldsymbol{h},\boldsymbol{n})$ of $\bQ_p^3$:
\beq\label{eq:rotazgeneric}
\cR_{\boldsymbol{n}}(\sigma)=\begin{pmatrix}
\frac{1-\alpha\sigma^2}{1+\alpha\sigma^2} & -\frac{2\alpha\sigma}{1+\alpha\sigma^2}&0\\ \frac{2\sigma}{1+\alpha\sigma^2} & \frac{1-\alpha\sigma^2}{1+\alpha\sigma^2}&0\\0&0&1
\end{pmatrix}
\eeq
with $\sigma\in\bQ_p\cup\{\infty\}$ and $\alpha=Q_+(\boldsymbol{h})/Q_+(\boldsymbol{g})$,
for every prime $p$.
{\hfill\qed}
\end{theor}

Eq. \eqref{eq:relationrot} can be written with respect to the canonical basis
$(\boldsymbol{e}_1,\boldsymbol{e}_2,\,\boldsymbol{e}_3)$ as
\beq\notag
  \boldsymbol{u}=M\cR_{\boldsymbol{n}}(\sigma)M^{-1}\boldsymbol{w},
\eeq
so that $M\cR_{\boldsymbol{n}}(\sigma)M^{-1}$ is the rotation matrix on $\bQ_p^3$ around $\boldsymbol{n}$ with respect to $(\boldsymbol{e}_1,\boldsymbol{e}_2,\,\boldsymbol{e}_3)$.

\begin{rem}
\label{remark:alpaherassiref}
Rotations around the reference axes of $\bQ_p^3$ are particular cases of Eq. \eqref{eq:rotazgeneric},
in which we choose the canonical basis as $(\boldsymbol{g},\boldsymbol{h},\boldsymbol{n})$
in some order. When $p\neq2$, if $(\boldsymbol{g},\boldsymbol{h},\boldsymbol{n})$ equals
\begin{itemize}
    \item $(\boldsymbol{e}_1,\boldsymbol{e}_2,\boldsymbol{e}_3)$, then $g=1,\ h=-v \Rightarrow \alpha=-v$,
    which determine a rotation around the third reference axis, $\cR_z(\sigma)$, characterized by
    \beq\notag
\begin{pmatrix}
\frac{1+v\sigma^2}{1-v\sigma^2} &
\frac{2v\sigma}{1-v\sigma^2}\\
\frac{2\sigma}{1-v\sigma^2} & \frac{1+v\sigma^2}{1-v\sigma^2}
\end{pmatrix};
\eeq
    \item $(\boldsymbol{e}_1,\boldsymbol{e}_3,\boldsymbol{e}_2)$, then $g=1,\ h=p\Rightarrow \alpha=p$,
    which determine a rotation around the second reference axis of the canonical basis,
    $\cR_y(\sigma)$, characterized by
    \beq\notag
\begin{pmatrix}
\frac{1-p\sigma^2}{1+p\sigma^2} &
-\frac{2p\sigma}{1+p\sigma^2}\\
\frac{2\sigma}{1+p\sigma^2} & \frac{1-p\sigma^2}{1+p\sigma^2}
\end{pmatrix};
\eeq
    \item $(\boldsymbol{e}_2,\boldsymbol{e}_3,\boldsymbol{e}_1)$,
    then $g=-v,\ h=p\Rightarrow \alpha=-\frac{p}{v}$, which determine
    a rotation around the first reference axis, $\cR_x(\sigma)$, characterized by
    \beq\notag
\begin{pmatrix}
\frac{1+\frac{p}{v}\sigma^2}{1-\frac{p}{v}\sigma^2} &
\frac{2\frac{p}{v}\sigma}{1-\frac{p}{v}\sigma^2}\\
\frac{2\sigma}{1-\frac{p}{v}\sigma^2} & \frac{1+\frac{p}{v}\sigma^2}{1-\frac{p}{v}\sigma^2}
\end{pmatrix}.
\eeq
\end{itemize}
These matrices are associated respectively to the conservation of the rank-2 quadratic forms $Q_{-v},\,Q_p$ and $-vy^2+pz^2$. The latter is $Q_{up}$ if $p\equiv1\mod4$, but it is again $Q_p$ for $p\equiv3\mod4$, in which case $\cR_x$ and $\cR_y$ are characterized by the same $2\times2$ submatrix, only differently positioned inside the $3\times3$ matrix on the canonical basis.

When $p=2$ we have $Q_+(\boldsymbol{e}_i)=1,\ i=1,2,3$, which gives
$\alpha=1$ whenever $(\boldsymbol{g},\boldsymbol{h},\boldsymbol{n})$
equals any permutation of the canonical basis. Then, a $2$-adic
rotation around any reference axis of $\bQ_2^3$ is characterized by
\beq\notag
\begin{pmatrix}
\frac{1-\sigma^2}{1+\sigma^2}&-\frac{2\sigma}{1+\sigma^2}\\ \frac{2\sigma}{1+\sigma^2}&\frac{1-\sigma^2}{1+\sigma^2}
\end{pmatrix}
\eeq
associated to the conservation of the rank-2 quadratic form $Q_1$.
\end{rem}
Eq. \eqref{eq:switchsign} translates in three dimensions into
\beq \label{eq:switchsign3}
\cR_{\boldsymbol{n}}\left(-\frac{1}{\alpha\sigma}\right)=\cR_{\boldsymbol{n}}\left(\infty\right)\cR_{\boldsymbol{n}}\left(\sigma\right).
\eeq

The important question we have to face now is, which of the three
(seven) equivalence classes of definite quadratic forms in dimension
two for odd $p$ ($p=2$), according to Proposition
\ref{prop:quadform2}, arise from restricting $Q_+$ to two-dimensional
subspaces $V = \boldsymbol{n}^\perp$ of $\QQ_p^3$. It will turn out
that all of them are realized. Since the transformations
$\cR_{\boldsymbol{n}}\in SO(3)_p$ are rotations around an axis
$\boldsymbol{n}\in\bQ_p^3$ preserving the rank-$3$ quadratic form
$Q_+$, this is equivalent to asking if the restrictions
$(\cR_{\boldsymbol{n}})_{\vert\boldsymbol{n}^\perp}$, by varying the
classes of $\boldsymbol{n}$, cover all the classes of special
orthogonal symmetries of the $p$-adic plane preserving the definite
rank-$2$ forms (recall that Theorem \ref{thm:rotation} identifies
all elements of $SO(3)_p$ as such rotations around axes).

Consider a basis change between orthogonal bases of $\bQ_p^3$ of the kind
$(\boldsymbol{g},\boldsymbol{h},\boldsymbol{n})
\mapsto (\lambda\boldsymbol{g},\mu\boldsymbol{h},\nu\boldsymbol{n})$, with
$\lambda,\mu,\nu\in\bQ_p^\ast$.
The parameter $\alpha$ in the parameterization \eqref{eq:rotazgeneric} is affected
by this transformation as
\beq\notag
  \alpha = \frac{Q_+(\boldsymbol{h})}{Q_+(\boldsymbol{g})}
        \mapsto \frac{Q_+(\mu\boldsymbol{h})}{Q_+(\lambda\boldsymbol{g})}
        =\left(\frac{\mu}{\lambda}\right)^2\frac{Q_+(\boldsymbol{h})}{Q_+(\boldsymbol{g})}.
\eeq
Hence, $\alpha$ is well-defined in $K=\bQ_p^\ast/(\bQ_p^\ast)^2$; there are at most
$4$ classes of rotations in $SO(3)_p$ for $p\neq2$ and $8$ classes for $p=2$.

In Corollary \ref{cor:forueedfapaval} we found a bijective relation
modulo squares between the value of the form $Q_+$ on a rotation axis
$\boldsymbol{n}$ and $\alpha$ \beq \label{eq:invforrelb}
\alpha\simeq \begin{cases}
               -vpQ_+(\boldsymbol{n}) &\text{ for } p \text{ odd},\\
                  Q_+(\boldsymbol{n}) &\text{ for } p=2.
             \end{cases}
\eeq
This allows us to determine the class of a rotation realized on a plane via
the class $Q_+(\boldsymbol{n})$ of its orthogonal axis.

As a consequence of the fact that $Q_+^{(4)}$ does not
represent $0$, Corollary \ref{cor:forueedfapaval} showed that $\alpha\not\simeq-1$.
Hence there are at most $3$ (or $7$) classes of rotations for odd $p$
(for $p=2$, respectively), depending on the equivalence class modulo squares of $\alpha$:
\beq\notag \begin{aligned}
&\alpha\in\{-v,p,up\},\ p>2,\\
&\alpha\in\{1,\pm2,\pm5,\pm10\}, \ p=2.
\end{aligned}
\eeq
Focusing on the language of quadratic forms, there are left at most $3$ (or $7$)
equivalence classes up to scaling of restrictions of $Q_+$ to a plane, identified
by the equivalence class modulo squares of $\alpha$.

We want to see if there exist classes of $\boldsymbol{n}$ in $\bQ_p^3$ realizing
all these classes of $\alpha$, uniquely associated to the definite rank-$2$
forms through their determinant. We make use of Eq. \eqref{eq:invforrelb}.

When $p$ is odd, $Q_+(\boldsymbol{n})\simeq -vp\alpha\Leftrightarrow
n_1^2-vn_2^2+pn_3^2+vp\alpha s^2=0$ for some $s\in\bQ_p^\ast$,
$\boldsymbol{n}$ of canonical coordinates $(n_1,n_2,n_3)$. The
quadratic form $n_1^2-vn_2^2+pn_3^2+vp\alpha s^2$ represents $0$ in
a nontrivial way for every $\alpha\in\{-v,p,up\}$, for its
determinant is $d\simeq-\alpha\not\simeq 1$ (Theorem
\ref{teor:Qrepns0}). If $n_1=n_2=n_3=0$, then $vp\alpha
s^2=0,\,s\in\bQ_p^\ast$ is impossible. It means that
$Q_+(\boldsymbol{n})\simeq-vp\alpha$ admits nontrivial solutions
$\boldsymbol{n}$ for every $\alpha$.

Similarly, when $p=2$,
$Q_+(\boldsymbol{n})\simeq\alpha\Leftrightarrow n_1^2+n_2^2+n_3^2-\alpha s^2=0$ for
some $s\in\bQ_p^\ast$. The form $x^2+y^2+z^2-\alpha s^2$ represents $0$ for every
$\alpha\in\{1,\pm2,\pm5,\pm10\}$, because its determinant is $d\simeq-\alpha\not\simeq 1$.
In summary, we have proved the following.

\begin{prop}
For every prime $p$ and every two-dimensional definite quadratic form $Q_\kappa$,
there exists a rotation axis determined by $\boldsymbol{n}\in\bQ_p^3\setminus\boldsymbol{0}$
such that $Q_{+\vert\boldsymbol{n}^\perp}\sim Q_\kappa$.
{\hfill\qed}
\end{prop}

In fact, when $p$ is odd,
\begin{itemize}
  \item $Q_+(\boldsymbol{e}_3)=p\ \Leftrightarrow\ \alpha\simeq-v$, as it can be seen from
     Remark \ref{remark:alpaherassiref} too. This is associated to the conservation of $Q_{-v}$;
  \item $Q_+(\boldsymbol{e}_2)=-v\ \Leftrightarrow\ \alpha\simeq p$, associated to $Q_p$;
  \item $Q_+(\boldsymbol{e}_1)=1\ \Leftrightarrow\ \alpha\simeq -\frac{p}{v}$.
     This is associated to $Q_{up}$ for $p\equiv1\mod4$ and $Q_p$ for $p\equiv3\mod4$.
\end{itemize}
On the other hand, when $p=2$, then $Q_+(\boldsymbol{e}_i)=1$ for all $i=1,2,3$, thus
$\alpha=1$, associated to the conservation of $Q_1$.

This means that when $p\equiv1\mod4$ the canonical basis is enough to realize all the
classes of rank-$2$ forms; however it leaves out the class associated to $\alpha\simeq up$
for $p\equiv3\mod4$, and it realizes only one class ($\alpha\simeq1$) out of the seven ones
for $p=2$. These remaining classes are covered by non-reference axes.
Indeed, $Q_+$ maintains some symmetry on $x,y$ for $p\equiv3\mod4$, contrary to
$p\equiv1\mod4$, and the symmetry is total for $p=2$.

The axes needed to realize all the classes of definite rank-$2$ forms when
$p\equiv3\mod4$ do not form an orthogonal basis. Moreover, in the case of $p=2$, we
need seven axes, so we cannot have a basis of $\bQ_2^3$ to realize all classes of
rank-$2$ forms.
%This is not a problem in view of a Cardano decomposition, where just a reference frame of $\bQ_p^3$ is needed, or even in a Euler-like representation where only two axes are needed. It is just a consequence of number theory, since in $\RR^3$ there are no distinct classes of axes or rotations.

Furthermore, the map $\boldsymbol{n}\mapsto
Q_{+\vert\boldsymbol{n}^\perp}$ from vectors $\boldsymbol{n}$ in
$\bQ_p^3\setminus\boldsymbol{0}$ to definite quadratic forms in
dimension two induces bijective correspondence between the
equivalence classes of $Q_+(\boldsymbol{n})$ in $K$ and the rank-two
definite forms $Q_\kappa$ up to equivalence. In fact, if
$\boldsymbol{n},\boldsymbol{n}'\in\bQ_p^3\setminus\boldsymbol{0}$
are such that $Q_+(\boldsymbol{n})\not\simeq Q_+(\boldsymbol{n}'),$ then
$Q_{+\vert\boldsymbol{n}^\perp}\not\simeq
Q_{+\vert{\boldsymbol{n}'}^\perp}$. This is because, as in the proof
of Proposition \ref{prop:numerattt}, the determinant is an invariant
of quadratic forms under basis changes, and $\det Q_+ \simeq
Q_+(\boldsymbol{g})Q_+(\boldsymbol{h})Q_+(\boldsymbol{n})
  \simeq Q_+(\boldsymbol{g}')Q_+(\boldsymbol{h}')Q_+(\boldsymbol{n}')$,
respectively on the orthogonal bases
$(\boldsymbol{g},\boldsymbol{h},\boldsymbol{n})$ and
$(\boldsymbol{g}',\boldsymbol{h}',\boldsymbol{n}')$.
As a consequence, if $Q_+(\boldsymbol{n})\not\simeq Q_+(\boldsymbol{n}')$, then
$Q_+(\boldsymbol{g})Q_+(\boldsymbol{h})\not\simeq Q_+(\boldsymbol{g}')Q_+(\boldsymbol{h}')$,
where
$Q_+(\boldsymbol{g})Q_+(\boldsymbol{h})
  \simeq\det Q_{+\vert\boldsymbol{n}^\perp} $,
$Q_+(\boldsymbol{g}')Q_+(\boldsymbol{h}')
  \simeq\det Q_{+\vert{\boldsymbol{n}'}^\perp}$.
Two quadratic forms $Q_{+\vert\boldsymbol{n}^\perp}$
and $Q_{+\vert{\boldsymbol{n}'}^\perp}$ with different determinants modulo
squares cannot be equivalent (Theorem \ref{teor:equivpQ}).

It is possible to choose representatives in the three (resp. seven) classes of
definite rank-$2$ forms for odd $p$ (reps. $p=2$) in such a way that if
$\boldsymbol{n},\boldsymbol{n}'\in\bQ_p^3$ satisfy $Q_+(\boldsymbol{n})\simeq Q_+(\boldsymbol{n}')$,
then $Q_{+\vert\boldsymbol{n}^\perp}\simeq Q_{+\vert{\boldsymbol{n}'}^\perp}$
by a linear transformation (scaling is not involved). This can be shown in
an abstract way, without resorting to the explicit classification of $p$-adic
quadratic forms. The statement is trivial if $\boldsymbol{n}=\lambda\boldsymbol{n}'$
for some $\lambda\in\bQ_p^\ast$. Otherwise, consider a vector $\boldsymbol{v}\in\bQ_p^3$
orthogonal to $\boldsymbol{n},\boldsymbol{n}'$. By Proposition \ref{prop:exrotiff},
there exists some $\cR_{\boldsymbol{v}}(\sigma)\in SO(3)_p$ such that
$\cR_{\boldsymbol{v}}(\sigma)\boldsymbol{n}=\lambda\boldsymbol{n}'$ with
$\lambda\in\bQ_p^\ast$.
Then, $\cR_{\boldsymbol{v}}(\sigma)$ transforms the plane $\boldsymbol{n}^\perp$
to the plane ${\boldsymbol{n}'}^\perp$. This is a linear map, which necessarily
implements the equivalence $Q_{+\vert\boldsymbol{n}^\perp}\simeq Q_{+\vert{\boldsymbol{n}'}^\perp}$

\section{$p$-adic Cardano and Euler decompositions}
\label{sec:CardEul2}
The previous development of $SO(3)_p$ showed a sufficiently close analogy
to the real Euclidean $SO(3)_{\RR}$, in that the three-dimensional space
has an essentially unique definite quadratic form, and all special orthogonal
transformations are actually rotations. The rotation groups around fixed axes
themselves turned out to be rather more complex: there are three (for odd $p$)
and seven (for $p=2$) types of rotation axes with associated different
rotation groups.

Now we ask if it is possible to express any element of $SO(3)_p$ as a composition
of rotations around the reference axes of $\bQ_p^3$, as it happens in $\RR^3$
according to Theorem \ref{teor:EulercompRot}.
In the following we will revert to calling the axes $x$, $y$ and $z$,
to make them typographically more distinguishable.

\begin{theor}[$p$-adic Cardano decomposition]
\label{teor:cardanoSO3p}
For every odd prime $p$, any $M\in SO(3)_p$ can be decomposed into
\beq\notag
  M=\cR_z(\zeta)\cR_y(\eta)\cR_x(\xi),
\eeq
for some parameters $\xi,\eta,\zeta\in\bQ_p\cup\{\infty\}$.
\end{theor}
\begin{proof}
Let $M\boldsymbol{e}_1$ be a vector with canonical coordinates $(m_1,m_2,m_3)$,
such that $Q_+(M\boldsymbol{e}_1)=Q_+(\boldsymbol{e}_1)=1$ since $M\in SO(3)_p$ preserves $Q_+$.

We need to find a composition of rotations around $\boldsymbol{e}_2$ and $\boldsymbol{e}_3$
transforming $M\boldsymbol{e}_1$ back to $\boldsymbol{e}_1$.

First, we show that there exists a $z$-rotation $\cR_z(\zeta)\in
SO(3)_p$ such that \beq\label{eq:rotprimaX}
\cR_z(\zeta)^{-1}M\boldsymbol{e}_1\in\text{span}(\boldsymbol{e}_1,\boldsymbol{e}_3).
\eeq This vector would be
$\cR_z(\zeta)^{-1}M\boldsymbol{e}_1=m_1'\boldsymbol{e}_1+m_3\boldsymbol{e}_3$
for some $m_1'\in\bQ_p$, where the third component of
$M\boldsymbol{e}_1$ is left unchanged by a rotation around
$\boldsymbol{e}_3$. As noted in Proposition \ref{prop:exrotiff},
necessary and sufficient condition for the existence of such a
$\cR_z(\zeta)^{-1}$ is \beq\notag
Q_+(m_1'\boldsymbol{e}_1+m_3\boldsymbol{e}_3)=Q_+(M\boldsymbol{e}_1)\Leftrightarrow
{m_1'}^2+pm_3^2=1. \eeq $Q_+(M\boldsymbol{e}_1)=m_1^2-vm_2^2+pm_3^2=1$
implies that $m_1,m_2,m_3\in\ZZ_p$ (cf. the compactness proof for
$SO(3)_p$ in Theorem \ref{thm:compactness}). Therefore, we resort to
Hensel's Lemma to show that $f(m_1')={m_1'}^2-1+pm_3^2$ admits roots
${m_1'}\in\bZ_p$. $f(m_1')\equiv{m_1'}^2-1\mod p$ has zeros
$m_1'\equiv\pm1\mod p$, in which the derivative
$f'(m_1')=2m_1'\not\equiv0\mod p$. Then, Hensel's Lemma allows us to
(uniquely) lift each of these solutions to a solution of the same
equation $\mod p^k$, converging to a $p$-adic solution $m_1'$. It
means that $\cR_z(\zeta)^{-1}$ as in Eq. \eqref{eq:rotprimaX}
exists.

Next, there exists a $y$-rotation $\cR_y(\eta)^{-1}\in SO(3)_p$ such that
\beq\label{eq:rote3e3}
\cR_y(\eta)^{-1}\cR_z(\zeta)^{-1}M\boldsymbol{e}_1=\boldsymbol{e}_1,
\eeq
because $Q_+\big(\cR_z(\zeta)^{-1}M\boldsymbol{e}_1\big)=Q_+(\boldsymbol{e}_1)$.

Eq. \eqref{eq:rote3e3} means that
$\cR_y(\eta)^{-1}\cR_z(\zeta)^{-1}M\in SO(3)_p$ has eigenvector
$\boldsymbol{e}_1$ with corresponding eigenvalue $1$, i.e.,
$\cR_y(\eta)^{-1}\cR_z(\zeta)^{-1}M:= \cR_x(\xi)$ is a rotation
around the $x$-axis.
\end{proof}

\begin{cor}
For every odd prime $p$, any $M\in SO(3)_p$ can be decomposed into
\beq\notag
\cR_z\cR_x\cR_y,\ \ \ \cR_x\cR_y\cR_z,\ \ \ \cR_y\cR_x\cR_z,
\eeq
respectively by certain parameters $\sigma,\,\tau,\omega\in\bQ_p\cup\{\infty\}$.
\end{cor}
\begin{proof}
To prove the existence of a decomposition of the kind $\cR_z\cR_x\cR_y$,
it is enough to repeat the steps of the proof of Theorem \ref{teor:cardanoSO3p}.
In particular, given
$M\boldsymbol{e}_2=m_1\boldsymbol{e}_1+m_2\boldsymbol{e}_2+m_3\boldsymbol{e}_3$,
there exists $\cR_z(\sigma)\in SO(3)_p$ such that
\beq\notag
\cR_z(\sigma)^{-1}M\boldsymbol{e}_2\in\text{span}(\boldsymbol{e}_2,\boldsymbol{e}_3).
\eeq
In fact, there exists $m_2'\in\bZ_p$ such that
$\cR_z(\sigma)^{-1}M\boldsymbol{e}_2=m_2'\boldsymbol{e}_2+m_3\boldsymbol{e}_3$ and
\beq\notag
Q_+(m_2'\boldsymbol{e}_2+m_3\boldsymbol{e}_3)=-v{m_2'}^2+pm_3^2=-v=Q_+(\boldsymbol{e}_2).
\eeq
To show this, we apply Hensel's Lemma to $f(m_2')=v{m_2'}^2-pm_3^2-v$:
$f(m_2')\equiv 0\mod p\Leftrightarrow v{m_2'}^2\equiv v\Leftrightarrow m_2'\equiv\pm1\mod p$,
and $f'(m_2')=2m_2'\not\equiv0\mod p$ in these solutions, which then are
lifted to $p$-adic solutions.

The existence of the decomposition with respect to the $x$, $y$ and $z$ axes
follows straightforwardly from Theorem \ref{teor:cardanoSO3p}:
if any $M^{-1}\in SO(3)_p$ can be written as $M^{-1}=\cR_z(\zeta)\cR_y(\eta)\cR_x(\xi)$
for certain parameters $\xi,\eta,\zeta$, then any $M\in SO(3)_p$ can be decomposed
into $M=\cR_x(\xi)^{-1}\cR_y(\eta)^{-1}\cR_z(\zeta)^{-1}:=\cR_x(\xi')\cR_y(\eta')\cR_z(\zeta')$,
where $\xi'=-\xi$ if $\xi\in\bQ_p$ or $\xi=\xi'$ are infinite, and similarly for $\eta',\zeta'$.

One similarly proves the existence of the nautical decomposition $\cR_y\cR_x\cR_z$ for
$SO(3)_p$ from the one as $\cR_z\cR_x\cR_y$.
\end{proof}

\begin{theor}[$p$-adic Cardano decomposition, $p\equiv 1$ mod 4]
\label{teor:cardanoSO3p:p=1mod4}
For every odd prime $p\equiv 1\mod4$, any $M\in SO(3)_p$ can be decomposed into
\beq\notag
  M=\cR_x(\zeta)\cR_z(\eta)\cR_y(\xi),
\eeq
for some parameters $\xi,\eta,\zeta\in\QQ_p\cup\{\infty\}$.

By applying this to $M^{-1}$, we get a decomposition $M=\cR_y(\zeta)\cR_z(\eta)\cR_x(\xi)$.
\end{theor}
\begin{proof}
The strategy is similar to the previous Theorem \ref{teor:cardanoSO3p}. We look
at $M\boldsymbol{e}_2 = m_1\boldsymbol{e}_1 + m_2\boldsymbol{e}_2 + m_3\boldsymbol{e}_3$,
which has $m_1^2-vm_2^2+pm_3^2 = Q_+(M\boldsymbol{e}_2) = Q_+(\boldsymbol{e}_2) = -v$.
We first look for a $\zeta$ with
$\cR_x(\zeta)^{-1}M\boldsymbol{e}_2 = m_1\boldsymbol{e}_1 + m_2'\boldsymbol{e}_2$,
which is guaranteed if $-v{m_2'}^2 = -vm_2^2+pm_3^2$, because in this
way $Q_+(m_1\boldsymbol{e}_1 + m_2'\boldsymbol{e}_2)=Q_+(M\boldsymbol{e}_2)=Q_+(\boldsymbol{e}_2)$,
and we can invoke Proposition \ref{prop:exrotiff}.
We prove that we can find $m_2'$ with ${m_2'}^2 = m_2^2 - \frac{p}{v}m_3^2$. Indeed,
note first that $m_2\not\equiv 0\mod p$, since otherwise $-v = Q_+(M\boldsymbol{e}_2) \equiv m_1^2 \mod p$,
which is impossible since $-v$ is a non-square in $\ZZ_p$ and hence modulo $p$
(it is here that we use $p\equiv 1\mod 4$, since then $-1$ is a square in $\QQ_p$,
thus $v$ and $-v$ are both non-squares). Thus, we have solutions $m_2' \equiv \pm m_2 \mod p$,
which are both nonzero. Since the derivative $2m_2'$ is then nonzero, as well, we
can invoke Hensel's Lemma to obtain a solution of the equation in $\ZZ_p$.

Now we proceed as before: since now
$\cR_x(\zeta)^{-1}M\boldsymbol{e}_2 \perp \boldsymbol{e}_3$, we can
find $\eta$ such that
$\cR_z(\eta)^{-1}\cR_x(\zeta)^{-1}M\boldsymbol{e}_2 =
\boldsymbol{e}_2$ (once more invoking Proposition
\ref{prop:exrotiff}). But this means that
$\cR_z(\eta)^{-1}\cR_x(\zeta)^{-1}M$ is a rotation around the
$y$-axis, i.e., $\cR_z(\eta)^{-1}\cR_x(\zeta)^{-1}M = \cR_y(\xi)$
for some $\xi$, and we are done.
\end{proof}

\begin{rem}
\label{oss:strctdimonecessaria}
Suppose a general Cardano or Euler type decomposition
$$
SO(3)_p\ni M
=\cR_{\boldsymbol{n}_1}(\sigma)\cR_{\boldsymbol{n}_2}(\tau)\cR_{\boldsymbol{n}_3}(\omega),
$$
where $\boldsymbol{n}_1\perp \boldsymbol{n}_2$ and
$\boldsymbol{n}_2\perp \boldsymbol{n}_3$ (in the Cardano
decomposition, $\boldsymbol{n}_1\perp \boldsymbol{n}_3$, while in
the Euler one $\boldsymbol{n}_1=\boldsymbol{n}_3$). We have
\begin{equation}\notag\begin{split}
 M=\cR_{\boldsymbol{n}_1}(\sigma)\cR_{\boldsymbol{n}_2}(\tau)\cR_{\boldsymbol{n}_3}(\omega)
    &\Leftrightarrow
 \cR_{\boldsymbol{n}_2}(\tau)^{-1}\cR_{\boldsymbol{n}_1}(\sigma)^{-1}M=\cR_{\boldsymbol{n}_3}(\omega) \\
    &\Leftrightarrow
 \cR_{\boldsymbol{n}_2}(\tau)^{-1}\cR_{\boldsymbol{n}_1}(\sigma)^{-1}M\boldsymbol{n}_3=\boldsymbol{n}_3.
\end{split}\end{equation}
Now, $\cR_{\boldsymbol{n}_2}(\tau)^{-1}$ preserves the component of the vector
$\cR_{\boldsymbol{n}_1}(\sigma)^{-1} M\boldsymbol{n}_3$ along $\boldsymbol{n}_2$,
which must be $0$ to get $\boldsymbol{n}_3$ as a result, therefore
$\cR_{\boldsymbol{n}_1}(\sigma)^{-1} M\boldsymbol{n}_3\perp \boldsymbol{n}_2$.

This shows that $\cR_{\boldsymbol{n}_1}(\sigma)^{-1} M\boldsymbol{n}_3\perp \boldsymbol{n}_2$
is a necessary condition for the existence of a decomposition
$M=\cR_{\boldsymbol{n}_1}(\sigma)\cR_{\boldsymbol{n}_2}(\tau)\cR_{\boldsymbol{n}_3}(\omega)$.
Since orthogonal transformations preserve $Q_+$, we also have
$Q_+\big(\cR_{\boldsymbol{n}_1}(\sigma)^{-1} M\boldsymbol{n}_3\big)=Q_+(\boldsymbol{n}_3)$.

Conversely, these conditions are also sufficient. To be precise, assume that
there exists a vector $\boldsymbol{v} \perp \boldsymbol{n}_2$ with
$\boldsymbol{n}_1^\top\boldsymbol{v} = \boldsymbol{n}_1^\top M\boldsymbol{n}_3$
and $Q_+(\boldsymbol{v}) = Q_+(M\boldsymbol{n}_3) = Q_+(\boldsymbol{n}_3)$.
Then, by Proposition \ref{prop:exrotiff} there exists $\sigma$ such that
$\cR_{\boldsymbol{n}_1}(\sigma)^{-1} M\boldsymbol{n}_3 = \boldsymbol{v} \perp \boldsymbol{n}_2$.
Note that after the choice of $\sigma$, it follows that $\tau$ and $\omega$ are
determined uniquely.
\end{rem}

\begin{rem}
The Cardano decompositions for $SO(3)_p$ of the kind $\cR_x\cR_z\cR_y$ and $\cR_y\cR_z\cR_x$
do not exist in general for odd primes $p\equiv 3\mod 4$.
In fact, one can construct counterexamples of matrices $M\in SO(3)_p$ for which there does
not exist $\cR_{\boldsymbol{n}_1}(\sigma)\in SO(3)_p$ such that
$\cR_{\boldsymbol{n}_1}(\sigma)^{-1} M\boldsymbol{n}_3\perp \boldsymbol{n}_2$ according
to Remark \ref{oss:strctdimonecessaria}, by exploiting
$Q_+\big(\cR_{\boldsymbol{n}_1}(\sigma)^{-1} M\boldsymbol{n}_3\big) = Q_+(\boldsymbol{n}_3)$.

As an example, let us show the existence of special orthogonal transformations
that cannot be written in the Cardano product form $\cR_x\cR_z\cR_y$
(the form $\cR_y\cR_z\cR_x$ follows a closely similar reasoning).
According to the recipe, we consider
$M\boldsymbol{e}_2 = m_1\boldsymbol{e}_1 + m_2\boldsymbol{e}_2 + m_3\boldsymbol{e}_3$,
which has $m_1^2+m_2^2+pm_3^2 = Q_+(M\boldsymbol{e}_2) = Q_+(\boldsymbol{e}_2) = 1$.
To get a corresponding Cardano decomposition, we would need $\sigma$ with
$\cR_x(\sigma)^{-1}M\boldsymbol{e}_2 = m_1\boldsymbol{e}_1 + m_2'\boldsymbol{e}_2$
such that ${m_2'}^2 = m_2^2+pm_3^2$, but this is impossible if $m_2=0$, since
$p$ is not a square in $\QQ_p$. It remains to show that this case can occur,
namely there exists an $M\in SO(3)_p$ with
$M\boldsymbol{e}_2 = m_1\boldsymbol{e}_1 + m_3\boldsymbol{e}_3$.
By Proposition \ref{prop:transitive}, it is sufficient to find solutions of
$m_1^2+pm_3^2=1$, which indeed has two roots for $m_1$, as we can see by first
solving it modulo $p$, resulting in $m_1 \equiv \pm 1\mod p$, and then
using Hensel's Lemma.
\end{rem}

\begin{rem}
{For odd primes $p$, none of the six possible Euler decompositions
for $SO(3)_p$ exist in general. We indicate how to construct counterexamples,
following  the same strategy as in the previous remark, based on
Remark \ref{oss:strctdimonecessaria}.}

(i \&{} ii) Nonexistence of XYX and YXY:
%$\cR_x\cR_y\cR_x$ and $\cR_y\cR_x\cR_y$:
Let us focus on the decomposition form $\cR_x\cR_y\cR_x$
(the form $\cR_y\cR_x\cR_y$ follows a closely similar reasoning).
According to the recipe, we consider
$M\boldsymbol{e}_1 = m_1\boldsymbol{e}_1 + m_2\boldsymbol{e}_2 + m_3\boldsymbol{e}_3$,
which has $m_1^2-vm_2^2+pm_3^2 = Q_+(M\boldsymbol{e}_1) = Q_+(\boldsymbol{e}_1) = 1$.
To get a corresponding Euler decomposition, we would need $\xi$ with
$\cR_x(\xi)^{-1}M\boldsymbol{e}_1 = m_1\boldsymbol{e}_1 + m_3'\boldsymbol{e}_3$
such that $p{m_3'}^2 = -vm_2^2+pm_3^2$, but this is impossible if $m_2\not\equiv 0\mod p$.
It remains to show that this case can occur, namely that there exists
an $M\in SO(3)_p$ with $m_2\not\equiv 0\mod p$.
By Hensel's Lemma, this boils down to finding solutions to
$m_1^2-vm_2^2\equiv 1\mod p$ in integers modulo $p$, such that $m_2\not\equiv 0\mod p$,
which always exist.
(Indeed, for $p\equiv 3\mod 4$, when $-v=1$, we can choose $m_2=1$.)

%We suspect that each of the other variants of Euler decomposition exhibits such counterexamples, too.

{
%The Euler decomposition for $SO(3)_p$ of the kind $\cR_x\cR_z\cR_x$ does not exist in
%general for odd primes $p$.
(iii) Nonexistence of XZX: According to the recipe, we consider
$M\boldsymbol{e}_1 = m_1\boldsymbol{e}_1 + m_2\boldsymbol{e}_2 + m_3\boldsymbol{e}_3$,
which has $m_1^2-vm_2^2+pm_3^2 = Q_+(M\boldsymbol{e}_1) = Q_+(\boldsymbol{e}_1) = 1$.
To get a corresponding Euler decomposition, we would need $\xi$ with
$\cR_x(\xi)^{-1}M\boldsymbol{e}_1 = m_1\boldsymbol{e}_1 + m_2'\boldsymbol{e}_2$,
such that $-v{m_2'}^2 = -vm_2^2+pm_3^2$, but this is impossible if $m_2=0$.
This case can occur: there exists an $M\in SO(3)_p$ with $m_2=0$, because $m_1^2+pm_3^2=1$
always admits solutions according to Hensel's Lemma.}

{
%Similarly, the Euler decomposition for $SO(3)_p$ of the kind $\cR_y\cR_z\cR_y$ does
%not exist in general for odd primes $p$.
(iv) Nonexistence of YZY:
Similarly,
consider $M\boldsymbol{e}_2 = m_1\boldsymbol{e}_1 + m_2\boldsymbol{e}_2 + m_3\boldsymbol{e}_3$,
which has $m_1^2-vm_2^2+pm_3^2 = Q_+(M\boldsymbol{e}_2) = Q_+(\boldsymbol{e}_2) = -v$.
To get a corresponding Euler decomposition, we would need $\eta$ with
$\cR_y(\eta)^{-1}M\boldsymbol{e}_2 = m_1'\boldsymbol{e}_1 + m_2\boldsymbol{e}_2$,
such that ${m_1'}^2 = m_1^2+pm_3^2$, but this is impossible if $m_1=0$.
However, there exists an $M\in SO(3)_p$ with $m_1=0$, because $-vm_2^2+pm_3^2=-v$
always admits solutions, as before using Hensel's Lemma.}

{
%Euler decompositions for $SO(3)_p$ of the kind $\cR_z\cR_x\cR_z$ and $\cR_z\cR_y\cR_z$
%do not exist in general for odd primes $p$. %$p\equiv1\mod4$.
(v \&{} vi) Nonexistence of ZXZ and ZYZ: Let us focus on
$\cR_z\cR_y\cR_z$ (the form $\cR_z\cR_x\cR_z$ follows a closely
similar reasoning), and consider $M\boldsymbol{e}_3 =
m_1\boldsymbol{e}_1 + m_2\boldsymbol{e}_2 + m_3\boldsymbol{e}_3$,
which has $m_1^2-vm_2^2+pm_3^2 = Q_+(M\boldsymbol{e}_3) =
Q_+(\boldsymbol{e}_3) = p$. To get a corresponding Euler
decomposition, we would need $\zeta$ with
$\cR_z(\zeta)^{-1}M\boldsymbol{e}_3 = m_1'\boldsymbol{e}_1 +
m_3\boldsymbol{e}_3$, such that ${m_1'}^2 = m_1^2-vm_2^2$. This ends
in a contradiction if the right hand side happens to be a non-square
in $\QQ_p$. Indeed, we will show that $m_1^2-vm_2^2=vp^2$ can occur.
To see this, consider first the condition $m_1^2-vm_2^2+pm_3^2 = p$,
from which we get $m_1^2-vm_2^2 \equiv 0 \mod p$, and since $v$ is a
non-square modulo $p$, this means that $m_1 \equiv m_2 \equiv 0 \mod
p$, i.e., we can write $m_1=px$, $m_2=py$ with $x,y\in\ZZ_p$, and
our condition after cancellation of $p$ factors becomes
$px^2-vpy^2+m_3^2=1$. First solving modulo $p$ and then using
Hensel's Lemma shows that this always has solutions for $m_3$,
regardless of the integers $x$ and $y$: $m_3 =
\pm\sqrt{1-px^2+vpy^2} \in \QQ_p$. This means we can choose $m_1$
and $m_2$ freely as multiples of $p$, and always satisfy the
quadratic form constraint. On the other hand, we claim that there
are solutions of $m_1^2-vm_2^2=vp^2$, which upon substituting $x$
and $y$ becomes $x^2-vy^2=v$. Indeed, the quadratic form on the left
hand side does not represent zero nontrivially, but $x^2-vy^2-vz^2$
does. So, we can take any solution of $x^2-vy^2-vz^2=0$, which
necessarily must have $z\neq 0$. By homogeneity, we can assume
$z=1$, which gives the desired solution.}
%
%Let us focus on $\cR_z\cR_x\cR_z$, and consider
%$M\boldsymbol{e}_3 = m_1\boldsymbol{e}_1 + m_2\boldsymbol{e}_2 + m_3\boldsymbol{e}_3$,
%which has $m_1^2-vm_2^2+pm_3^2 = Q(M\boldsymbol{e}_3) = Q(\boldsymbol{e}_3) = p$.
%To get a corresponding Euler decomposition, we would need $\zeta$ with
%$\cR_z(\zeta)^{-1}M\boldsymbol{e}_3 = m_2'\boldsymbol{e}_2 + m_3\boldsymbol{e}_3$
%such that $-v{m_2'}^2 = m_1^2-vm_2^2$, but this is impossible if $m_2=0$ when
%$p\equiv1\mod4$, in which case $-v$ is not a square in $\QQ_p$.
%It remains to show that $m_1^2+pm_3^2=p$ has solutions. The quadratic form
%$m_1^2+pm_3^2$ does not represent zero, while $m_1^2+pm_3^2-ps^2,\ s\in\QQ_p^\ast$,
%represents zero nontrivially. It follows that $m_1^2+pm_3^2=p$ has nontrivial solutions.
\end{rem}

\begin{rem}
When $p=2$, none of the Cardano and Euler decompositions exist for all of $SO(3)_2$.
The counterexample is a single matrix
\begin{equation}\notag
  M = \begin{pmatrix}
        -2                   & -2                   & \sqrt{-7}\\
        \frac12(\sqrt{-7}+1) & \frac12(\sqrt{-7}-1) & 2\\
        \frac12(\sqrt{-7}-1) & \frac12(\sqrt{-7}+1) & 2
      \end{pmatrix}.
\end{equation}
Note that $-7$ is a square in $\QQ_2$, and indeed $\sqrt{-7}=1+2^2+2^4+2^5+2^7+\ldots$,
so that $\frac12(\sqrt{-7}+1) = 1+2+2^3+2^4+2^6+\ldots$ is odd,
and $\frac12(\sqrt{-7}-1) = 2+2^3+2^4+2^6+\ldots$ is even in $\ZZ_2$.

As a matter of fact, $M$ does not have a decomposition of any of the six Cardano
and six Euler forms shown in Theorem \ref{teor:EulercompRot}.
Let us show that it cannot be written as a product
$\cR_x\cR_y\cR_z$. Namely, by Remark \ref{oss:strctdimonecessaria}
it is enough to show that there does not exist a $\xi$ such that
$\cR_x(\xi)^{-1} M\boldsymbol{e}_3 \perp \boldsymbol{e}_2$. With
$M\boldsymbol{e}_3 = m_1\boldsymbol{e}_1 + m_2\boldsymbol{e}_2 +
m_3\boldsymbol{e}_3$ such that $m_1^2+m_2^2+m_3^2=1$, we would need
a vector $m_1\boldsymbol{e}_1+m_3'\boldsymbol{e}_3$ with quadratic
form evaluating to $1$, i.e., ${m_3'}^2 = m_2^2+m_3^2 = 8$, which
however is not a square in $\QQ_2$; contradiction.

Likewise, it cannot be decomposed in the form $\cR_y\cR_x\cR_z$,
because then we would need $\cR_y(\eta)^{-1} M\boldsymbol{e}_3 \perp
\boldsymbol{e}_1$, which amounts to a vector
$m_2\boldsymbol{e}_2+m_3'\boldsymbol{e}_3$ with quadratic form
evaluating to $1$, i.e., ${m_3'}^2 = m_1^2+m_3^2 = -3$, which
however is not a square in $\QQ_2$, as it is $\equiv 5 \mod 8$;
contradiction.

Any other decomposition results in a contradiction of the same type, because,
as can be checked, every column of $M$ has the property that, while
the squares of its elements sum to $1$, the sum of any two squares is not a
square in $\QQ_2$. To prove this, it is enough to consider all occurring
cases and check that the sum in question is not a square modulo $8$ or $16$.
\end{rem}

%%%%%%%%%%%%%%%%%%%%%%%%%%%%%%%%%%%%%%%%%%%%%%%%%%%%%%%%%%%%%%%%%%%%%%%%%%%%%%%%
\begin{comment}
However, the above argument can be turned around $SO(3)_2$ is the union of different
products of rotations around $x$, $y$ and $z$ in the different forms of Cardano and
Euler decomposition. More precisely, the following holds.

\begin{theor}[Modified Cardano decomposition for $\mathbf{p=2}$]
\label{thm:Euler-Cardano:p=2}
...
\end{theor}
\begin{proof}
...
\end{proof}
\end{comment}
%%%%%%%%%%%%%%%%%%%%%%%%%%%%%%%%%%%%%%%%%%%%%%%%%%%%%%%%%%%%%%%%%%%%%%%%%%%%%%%%

%\medskip
Although these results are in contrast to the Euclidean case, $SO(3)_p$ is still
generated by its subgroups $G_x,\,G_y,\,G_z$ of rotations around the reference axes,
at least for all odd primes $p$.
Now we will find the multiplicity of the above Cardano representations of $SO(3)_p$
for odd $p$.

\begin{prop}
If $M\in SO(3)_p$ and $M=\cR_x(\xi)\cR_y(\eta)\cR_z(\zeta)$, then
there exists at least another distinct Cardano representation of $M$
along the same axes \beq\notag
  M=\cR_x(\infty)\cR_x(\xi)\cR_y\left(\frac{1}{\alpha\eta}\right)\cR_z(\infty)\cR_z(\zeta).
\eeq
\end{prop}
\begin{proof}
A systematic ambiguity of order $2$ in the products $\cR_x(\xi)\cR_y(\eta)\cR_z(\zeta)$
is based on the relation
\beq\notag
  \cR_x(\infty)\cR_y(\infty)\cR_z(\infty)=I.
\eeq
By using it with Eq. \eqref{eq:switchsign3} we get
\beq\notag\begin{aligned}
M&=\cR_x(\xi)\cR_y(\eta)\cR_z(\zeta)\\
&=\cR_x(\xi)\cR_y(\eta)\cR_x(\infty)\cR_y(\infty)\cR_z(\infty)\cR_z(\zeta)\\
&= \cR_x(\infty)\cR_x(\xi)\big(\cR_y(\eta)+\cR_x(\infty)\left[\cR_y(\eta),\cR_x(\infty)\right]\big)\cR_y(\infty)\, \cR_z(\infty)\cR_z(\zeta)\\
&=\cR_x(\infty)\cR_x(\xi)\begin{pmatrix}
-e(\eta)&0&f(\eta)\\
0&1&0\\
g(\eta)&0&-e(\eta)
\end{pmatrix}\cR_z(\infty)\cR_z(\zeta),
\end{aligned}\eeq
where the product of $\cR_x(\infty)$ with the commutator of matrices $\left[\cR_y(\eta),\cR_x(\infty)\right]$ provides the transformation $\eta\mapsto-\eta$ (an infinite parameter remains the same) on the parameter of the $y$-rotation (sign change of the off-diagonal elements), which together with $\cR_y(\infty)$ globally gives $\eta\mapsto\frac{1}{\alpha\eta}$ (sign change of the diagonal entries).
\end{proof}

We give two elementary results in order to prove that the Cardano representation
of $SO(3)_p$ along the axes $x$, $y$ and $z$ is exactly twofold for odd $p$.
\begin{prop}
\label{prop:coppieuguali}
For odd $p$, and all $\cR_x(\xi)$, $\cR_x(\xi')$, $\cR_y(\eta)$, $\cR_y(\eta')\in SO(3)_p$,
\beq\notag
\cR_x(\xi)\cR_y(\eta)= \cR_x(\xi')\cR_y(\eta')\Leftrightarrow \cR_x(\xi)= \cR_x(\xi'),\, \cR_y(\eta)=\cR_y(\eta').
\eeq
\end{prop}
\begin{proof}
This is immediate to prove by equating two matrices of the kind
\beq\notag
\cR_x(\xi)\cR_y(\eta)=\begin{pmatrix}
1&0&0\\0&a(\xi)&b(\xi)\\0&c(\xi)&a(\xi)
\end{pmatrix}\begin{pmatrix}
e(\eta)&0&f(\eta)\\
0&1&0\\
g(\eta)&0&e(\eta)
\end{pmatrix}=\begin{pmatrix}e(\eta)&0&f(\eta)\\b(\xi)g(\eta)&a(\xi)&b(\xi)e(\eta)\\a(\xi)g(\eta)&c(\xi)&a(\xi)e(\eta)\end{pmatrix}
\eeq
thanks to the fact that $e(\eta),\,a(\xi)\neq0$ for every parameter.
\end{proof}

\begin{cor}
\label{cor:purtserve}
Let $\cR_x(\xi)\cR_y(\eta)\cR_{\boldsymbol{n}}(\sigma)= \cR_x(\xi')\cR_y(\eta')\cR_{\boldsymbol{n}}(\sigma')$, $p>2$. If  $\big(\cR_x(\xi),\,\cR_y(\eta),\,\cR_{\boldsymbol{n}}(\sigma)\big)\neq\big(\cR_x(\xi'),\,\cR_y(\eta'),\,\cR_{\boldsymbol{n}}(\sigma')\big)$ then $\cR_{\boldsymbol{n}}(\sigma)\neq\cR_{\boldsymbol{n}}(\sigma')$.
\end{cor}
\begin{proof}
We prove the contrapositive: if $\cR_{\boldsymbol{n}}(\sigma)= \cR_{\boldsymbol{n}}(\sigma')$, then
\beq\notag %\begin{aligned}
\cR_x(\xi)\cR_y(\eta)\cR_{\boldsymbol{n}}(\sigma)=
\cR_x(\xi')\cR_y(\eta')\cR_{\boldsymbol{n}}(\sigma')\Rightarrow
\cR_x(\xi)\cR_y(\eta)=\cR_x(\xi')\cR_y(\eta').
%\end{aligned}
\eeq
By Proposition \ref{prop:coppieuguali}, this is equivalent to $\big(\cR_x(\xi),\,\cR_y(\eta),\,\cR_{\boldsymbol{n}}(\sigma)\big)=\big(\cR_x(\xi'),\,\cR_y(\eta'),\,\cR_{\boldsymbol{n}}(\sigma')\big)$.
\end{proof}

\begin{theor}
Every $M\in SO(3)_p$, for odd prime $p$, has exactly two distinct
Cardano decompositions with respect to the $x$, $y$ and $z$ axes
\beq\notag
  M=\cR_x(\xi)\cR_y(\eta)\cR_z(\zeta)
\eeq
for some $\xi,\eta,\zeta\in\bQ_p\cup\{\infty\}$, and
\beq\notag
M=\cR_x(\infty)\cR_x(\xi)\cR_y\left(\frac{1}{\alpha\eta}\right)\cR_z(\infty)\cR_z(\zeta).
\eeq
\end{theor}
\begin{proof}
We look for nontrivial solutions of
$\cR_x(\xi)\cR_y(\eta)\cR_z(\zeta)=\cR_x(\xi')\cR_y(\eta')\cR_z(\zeta')$.
Using the parameterization \eqref{eq:rotazgeneric}, we have
\begin{equation}\notag\begin{split}
\cR_x(\xi)\cR_y(\eta)\cR_z(\zeta) &=
\begin{pmatrix}
  1&0&0\\
  0&a(\xi)&b(\xi)\\
  0&c(\xi)&a(\xi)
\end{pmatrix}
\begin{pmatrix}
  e(\eta)&0&f(\eta)\\
  0&1&0\\
  g(\eta)&0&e(\eta)
\end{pmatrix}
\begin{pmatrix}
  l(\zeta)&m(\zeta)&0\\
  n(\zeta)&l(\zeta)&0\\
  0&0&1
\end{pmatrix} \\
&=\begin{pmatrix}
    e(\eta)l(\zeta)&e(\eta)m(\zeta)&f(\eta)\\
    a(\xi)n(\zeta)+b(\xi)g(\eta)l(\zeta)&a(\xi)l(\zeta)+b(\xi)g(\eta)m(\zeta)&b(\xi)e(\eta)\\
    c(\xi)n(\zeta)+a(\xi)g(\eta)l(\zeta)&c(\xi)l(\zeta)+a(\xi)g(\eta)m(\zeta)&a(\xi)e(\eta)
  \end{pmatrix}.
\end{split}\end{equation}
We equate two matrices of this kind, whose entries we call $m_{ij}$ and
$m_{ij}'$ respectively ($i,j=1,2,3$). Thus,
\beq\notag
  m_{13}=m_{13}'\Leftrightarrow f(\eta)=f(\eta'),
\eeq
equivalent to $-\frac{2p\eta}{1+p\eta^2}=-\frac{2p\eta'}{1+p{\eta'}^2}\Leftrightarrow 2p(\eta-\eta')(1-p\eta\eta')=0\Leftrightarrow \eta=\eta'$ or $\eta'=1/(p\eta)$.
Furthermore, $f(\eta)=f(\eta')\Leftrightarrow g(\eta)=g(\eta')$, since
$g(\eta)=\frac{2\eta}{1+p\eta^2}$.

If $\eta=\eta'\in\bQ_p\cup\{\infty\}$, then $\cR_y(\eta)=\cR_y(\eta')$,
otherwise $\eta'=1/(p\eta)\in\bQ_p\cup\{\infty\}$ and
\beq\notag
\cR_y(\eta) =
\begin{pmatrix}
  -e(\eta')&0&f(\eta')\\0&1&0\\ g(\eta')&0&-e(\eta')
\end{pmatrix}.
\eeq
As a consequence, from $m_{11}=m_{11}'$, $m_{12}=m_{12}'$, $m_{23}=m_{23}'$, $m_{33}=m_{33}'$
we deduce $l(\zeta)=\pm l(\zeta')$, $m(\zeta)=\pm m(\zeta')$, $a(\xi)=\pm a(\xi')$,
$b(\xi)= \pm b(\xi')$, where the $+$ sign is always related to $\eta=\eta'$ and the $-$ sign
to $\eta\eta'=1/p$. This is because $e(\eta)\neq0$ for every $\eta\in\bQ_p\cup\{\infty\}$,
as $e(\eta)=\frac{1-p\eta^2}{1+p\eta^2}=0\Leftrightarrow \eta^2=1/p$, but
$1/p$ is not a square; $e(\infty)=-1$.

Then, $m_{22}=m_{22}'$ is always satisfied, whereas
\beq\notag\begin{aligned}
  m_{21}=m_{21}' &\Leftrightarrow a(\xi)n(\zeta)+b(\xi)g(\eta)l(\zeta) = a(\xi')n(\zeta'
                                   +b(\xi')g(\eta')l(\zeta')\\
                 &\Leftrightarrow a(\xi)n(\zeta)=\pm a(\xi)n(\zeta')\
                  \Leftrightarrow n(\zeta)=\pm n(\zeta'); \\
  m_{31}=m_{32} &\Leftrightarrow c(\xi)n(\zeta)+a(\xi)g(\eta)l(\zeta)=c(\xi')n(\zeta')
                                  +a(\xi')g(\eta')l(\zeta')\\
                &\Leftrightarrow c(\xi)n(\zeta)=\pm c(\xi')n(\zeta)
                 \Leftrightarrow c(\xi)=\pm c(\xi').
\end{aligned}
\eeq
Note that $a(\xi)\neq0$ for every $\xi\in\bQ_p\cup\{\infty\}$, because
$a(\xi)=\frac{1+\frac{p}{v}\xi^2}{1-\frac{p}{v}\xi^2}=0\Leftrightarrow \xi^2=-v/p$
but $-v/p$ is not a square; $a(\infty)=-1$.

In the last step it might be
$n(\zeta)=\frac{2\zeta}{1-v\zeta^2}=0\Leftrightarrow \zeta=0$, for which the
condition $m(\zeta)=\pm m(\zeta')=0$ gives $\cR_z(\zeta)=\cR_z(\zeta')=I$.
This implies $\cR_x(\xi)=\cR_x(\xi'),\,\cR_y(\eta)=\cR_y(\eta')$ by
Corollary \ref{cor:purtserve}, in particular $c(\xi)=c(\xi')$, in agreement
with what we deduced above.
Then, $m_{32}=m_{32}'$ is satisfied.

We have found that $\cR_x(\xi)\cR_y(\eta)\cR_z(\zeta)=\cR_x(\xi')\cR_y(\eta')\cR_z(\zeta')$
if and only if either
\beq\notag
\bigl(\cR_x(\xi),\cR_y(\eta),\cR_z(\zeta)\bigr) = \bigl(\cR_x(\xi'),\cR_y(\eta'),\cR_z(\zeta')\bigr)
\eeq
or
\beq\notag
\left\{\begin{aligned}
&\cR_x(\xi)=\cR_x(\infty)\cR_x(\xi'), \\ &\cR_y(\eta)=\begin{pmatrix}
-e(\eta')&0&f(\eta')\\
0&1&0\\
g(\eta')&0&-e(\eta')
\end{pmatrix}, \\
&\cR_z(\zeta)=\cR_z(\infty)\cR_z(\zeta'),
\end{aligned}\right.
\eeq
which concludes the proof.
\end{proof}

\begin{rem}
The above results of the duplicity of the Cardano decomposition can be
understood qualitatively, and for any $p$, from Remark \ref{oss:strctdimonecessaria}.

Indeed, assume that $SO(3)_p \ni M
=\cR_{\boldsymbol{n}_1}(\sigma)\cR_{\boldsymbol{n}_2}(\tau)\cR_{\boldsymbol{n}_3}(\omega)$,
where $\boldsymbol{n}_1 \perp \boldsymbol{n}_2 \perp \boldsymbol{n}_3$.
We have seen that this is possible if and only if
$\cR_{\boldsymbol{n}_1}(\sigma)^{-1} M\boldsymbol{n}_3\perp \boldsymbol{n}_2$,
and chosen $\cR_{\boldsymbol{n}_1}(\sigma)$, the other two rotations are
uniquely determined. There are at most two solutions, since there is
only one free parameter in conditions for
$\boldsymbol{v} = \cR_{\boldsymbol{n}_1}(\sigma)^{-1} M\boldsymbol{n}_3$,
and they boil down to a single quadratic equation.
On the other hand, with $\cR_{\boldsymbol{n}_1}(\sigma)$,
another solution of this condition is always
$\cR_{\boldsymbol{n}_1}(\infty)\cR_{\boldsymbol{n}_1}(\sigma)$.
\end{rem}

\section{Discussion}
\label{sec:discussion}
We have seen that not only is there a well-motivated $p$-adic special orthogonal
group $SO(3)_p$, but that it shares many geometric features with its real
counterpart $SO(3)_{\RR}$, although there are also crucial
differences, owing to the specific number theory of $\ZZ_p$ for different
primes $p$, in particular distinguishing odd $p$ from $p=2$, and within
odd primes between $p\equiv 1\mod4$ and $p\equiv 3\mod4$.
These differences manifested themselves most blatantly in the discussion
of Cardano and Euler angle decompositions of general special orthogonal
transformations in terms of rotations around the reference axes:
they are not available in general for $p=2$, but certain (all) Cardano
decompositions hold for odd primes $p\equiv 3\mod 4$ ($p\equiv 1\mod 4$).
We leave open the possibility of modified principal angle decompositions
beyond the use of the reference axes.

Note that we have treated the groups essentially as algebraic groups, and
in future development, it may pay off to follow the analytic approach
of $p$-adic Lie groups and Lie algebras \cite{Serre:Lie}, which allows for a
local description in terms of infinitesimal generators.

An open question that we have left, might be answerable in this way, which
is that after the cyclicity of the two-dimensional rotation groups $SO(2)_p^\kappa$.
Since the groups are definitely profinite, the correct question seems to
be whether the groups $SO(2)_p^\kappa$ are procyclic. To prove this,
we consider their projections modulo $p^k$,
$\pi_k\bigl(SO(2)_p^\kappa\bigr) \subset SL(2,\ZZ/p^k\ZZ)$, and
should show that for every $p$ and $\kappa$, and sufficiently large $k$
these are cyclic, cf.~\cite{Ilaria:tesi}.

Finally, we want to discuss our motivation for considering $p$-adic
spatial rotations in the first place: it grows out of Volovich's
$p$-adic quantum theory, in which Euclidean space is replaced by
$p$-adic space to define the underlying phase space \cite{Volovich},
see also \cite{V+V,Varadarajan,Vuvuzela}. The idea is to realise
quantum systems as unitary representations of the symmetry group of
the $p$-adic space, according to Noether's theorem \cite{Weyl}, and
this has been realised for the displacement operations in position
and momentum, resulting in a $p$-adic Heisenberg--Weyl algebra of
position and momentum operators. What has not been done yet is to
develop a quantum theory of $p$-adic angular momentum. By the same
philosophy, this is identical to the classification of all the
projective unitary irreducible representations of $SO(3)_p$. We get
infinitely many such representations from reducing the group modulo
$p^k$, noting that $SO(3)_p \mod p^k$ are finite groups, for which
the irreps can be found by standard tools \cite{FultonHarris}.
Indeed, in \cite{Michele:tesi,Ilaria:tesi}, this has been done for
reduction modulo $p$ and $p^2$, for certain odd primes. An open
question is whether all irreps of $SO(3)_p$ arise in this way. We
suspect that this problem cannot be solved using standard $p$-adic
techniques based on Hensel's Lemma, but may require non-standard
$p$-adic analysis and Gretel's Lemma \cite{Grimm}.

\section*{Acknowledgments}
This project originated with the research visits of IS and MP with the
Grup d'Informaci\'o Qu\`antica at UAB, Barcelona, in the context of their
degree and MSc projects.
MP and IS were co-funded by the Erasmus+ programme of the European Union.
AW acknowledges support by the EU (STREP ``RAQUEL''), the ERC (AdG ``IRQUAT''),
the Spanish MINECO (grants FIS2013-40627-P and FIS2016-86681-P)
with the support of FEDER funds, as well as by the Generalitat de
Catalunya CIRIT, projects 2014-SGR-966 and 2017-SGR-112.

%\bibliographystyle{unsrt}
%\bibliography{fqsw}

\end{document}